\documentclass[12pt,draftcls,onecolumn]{IEEEtran}

\IEEEoverridecommandlockouts   
\usepackage{amsmath}
\usepackage{amsfonts}
\usepackage{amssymb}

\usepackage{amsthm}
\usepackage{setspace}
\usepackage{subfigure}
\usepackage{color}
\usepackage{float}
\usepackage{graphicx}
\usepackage{lettrine}
\usepackage{pgfplots}

\usepackage{algorithm}
\usepackage{algpseudocode}

\usepackage{hyperref}

\newtheorem{thm}{Theorem}
\newtheorem{lemma}{Lemma}

\newtheorem{defin}{Definition}

\usepackage{enumerate}

\newcommand{\bmtx}{\begin{bmatrix}}
\newcommand{\emtx}{\end{bmatrix}}
\newcommand{\bsmtx}{\left[ \begin{smallmatrix}} 
\newcommand{\esmtx}{\end{smallmatrix} \right]}
\newcommand{\bmatarray}[1]{\left[\begin{array}{#1}}
\newcommand{\ematarray}{\end{array}\right]}

\newcommand{\field}[1]{\mathbb{#1}}
\newcommand{\R}{\field{R}}
\newcommand{\N}{\field{N}}
\newcommand{\Z}{\field{Z}}

\newcommand{\C}{\field{C}}

\newcommand{\Dc}{\mathcal{D}^c}

\begin{document}

\title{Robust Regret Optimal Control}

\author{Jietian Liu and Peter Seiler% <-this % stops a space
  \thanks{Jietian Liu and Peter Seiler are with the Department of
    Electrical Engineering and Computer Science, University of
    Michigan, Emails: {\tt\small jietian@umich.edu} and
    {\tt\small pseiler@umich.edu}.  } }

\date{}
\maketitle

%%%%%%%%%%%%%%%%%%%%%%%%%%%%%%%%%%%%%%%%%%%%%%%%%%%%%%%%%%%%%%%%%%%%%%
\begin{abstract}

  This paper presents a synthesis method for robust, regret optimal
  control.  The plant is modeled in discrete-time by an uncertain
  linear time-invariant (LTI) system.  An optimal non-causal
  controller is constructed using the nominal plant model and given
  full knowledge of the disturbance.  Robust regret is defined
  relative to the performance of this optimal non-causal control.  It
  is shown that a controller achieves robust regret if and only if it
  satisfies a robust $H_\infty$ performance condition. DK-iteration
  can be used to synthesize a controller that satisfies this condition
  and hence achieve a given level of robust regret. The approach is
  demonstrated three examples: (i) a simple single-input,
  single-output classical design, (ii) a longitudinal control for a
  simplified model for a Boeing 747 model, and (iii) an active
  suspension for a quarter car model. All examples compare the robust
  regret optimal against regret optimal controllers designed without
  uncertainty.

  % Mention that the definition of robust regret includes both
  % competitive ratio and (additive) regret?

\end{abstract}

% -----------------------------------------------------------
\section{Introduction}
\label{sec:intro}

This paper considers control design for a discrete-time, linear
time-invariant (LTI) plant. The goal is to design an output-feedback
controller to: (i) stabilize the plant, and (ii) ensure a generalized
error remains small in the presence of disturbances.  Two common
problem formulations assume the plant is known and minimize the
closed-loop gain from disturbance to error in either the $H_2$ or
$H_\infty$ norm~\cite{zhou96,dullerud99}. The disturbance in the $H_2$
or $H_\infty$ optimal control problems has the interpretation of being
random noise or the worst-case (antagonistic) input, respectively.

Recent work has focused on alternative formulations that avoid making
these assumptions on the disturbance. Two threads of important results
are most relevant for our paper.  The first thread considers
performance relative to the optimal non-causal (dynamic) controller
with full knowledge of the disturbance
\cite{goel20arXiv,goel19PMLR,goel22CDC,goel21PMLR,sabag21ACC,sabag21arXiv}.
Regret is defined to be the (additive) difference in performance
achieved by a given controller and the non-causal controller.  The
objective is to design a controller that minimizes the worst-case
regret over all disturbances. It is shown that this problem can be
converted to an equivalent $H_\infty$ synthesis problem with a scaled
plant model. Subsequent work considered the multiplicative ratio in
performance, called the competitive ratio, achieved by a given
controller relative to the non-causal controller
\cite{goel22TAC,sabag22arXiv}.  Additional related work considers
state-and input constraints using system level
synthesis~\cite{didier22} and regret-bounds for $H_\infty$
controllers~\cite{karapetyan22}. A second thread measures regret of a
given controller relative to the best static state-feedback with full,
non-causal knowledge of the disturbance sequence
\cite{agarwal19}. Online convex optimization techniques \cite{hazan16}
are used to optimize over a class of disturbance action policies that
depend on a finite history of the disturbance.
 
The key contribution of our paper is to provide a solution to a robust
regret optimal control problem.  We formulate the output feedback
control problem in Section~\ref{sec:robprob} with an uncertain plant
and general interconnection used in the robust control literature
\cite{zhou96,dullerud99}.  We define robust regret relative to the
performance achieved by the optimal non-causal controller on the
nominal plant (with no uncertainty).  We use a definition of regret
that includes as special cases $H_\infty$
control~\cite{zhou96,dullerud99,dgkf89}, (additive) regret
\cite{goel20arXiv,goel19PMLR,goel22CDC,goel21PMLR,sabag21ACC,sabag21arXiv}
and (multiplicative) competitive ratio
\cite{goel22TAC,sabag22arXiv}. Our definition of regret is a special
case of regret-optimal control with weights~\cite{sabag22arXiv}
(although we include model uncertainty). It is important to note that
model uncertainty is not equivalent to an exogenous disturbance. For
example, model uncertainty can cause an instability and unbounded
signals but this is not possible with bounded disturbances. This
distinct feature of model uncertainty motivates the importance of
considering robustness in addition to the disturbance rejection.

% This distinct feature of model uncertainty motivates the approach
% given in this paper.

% Our definition of regret is more
% general than in previous works and includes as special cases the
% (additive) regret defined in
% \cite{goel20arXiv,goel19PMLR,goel22CDC,goel21PMLR,sabag21ACC,sabag21arXiv}
% and (multiplicative) competitive ratio in \cite{goel22TAC}. 

We solve the robust regret problem using a similar solution approach
as in these prior works. First we derive the optimal non-causal
controller from \cite{hassibi99} in our more general setting
(Section~\ref{sec:K0}). Second, the robust regret problem is
converted, via a spectral factorization of the optimal non-causal cost,
% (Appendix~\ref{app:DTSF}),
to an equivalent robust synthesis problem
(Section~\ref{sec:robfeas}). The robust synthesis problem is
non-convex but a sub-optimal controller can be computed via a
coordinate-wise search known as DK-iteration or $\mu$-synthesis
\cite{doyle85,doyle87,balas94,lind94,packard93asme}. As an
intermediate step, we solve the nominal regret problem in
Section~\ref{sec:nomregret} for the case of known plant dynamics. As
noted above, the nominal case generalizes the cost function and
definition of regret compared to previous works. The nominal case also
forms the foundation for our main results on robust regret with model
uncertainty.  Section~\ref{sec:example} provides examples to compare
the proposed method against existing robust control methods,
e.g. $H_\infty$ control, and (nominal) regret-based methods,
e.g. additive regret and competitive ratio.

It is important to note that regret is typically defined using an
omniscient control design as a benchmark.  Our nominal regret
definition applies this convention to the disturbance, i.e. the
benchmark is the optimal non-causal controller with full knowledge of
the disturbance. However, our robust regret definition does not apply
this convention to the model uncertainty. Specifically, our benchmark
for robust regret is the optimal controller designed on the nominal
model with $\Delta=0$.  Thus the method in this paper designs
controllers that have small regret with respect to the disturbance and
are robust to model variations.  It may be possible to design robust
regret optimal controllers where the baseline has full knowledge of
the disturbance and the specific value of the uncertain plant, e.g.
by adapting gain-scheduling results such as \cite{packard94}.  This
will be considered in future work.

% \cite{abbasi11}: This is one of the earliest works (the first?) on
% regret-optimal control.
% \cite{hassibi99}: Solution of optimal non-causal controller.
% \cite{sabag21ACC,sabag21arXiv}: Regret-optimal with full information.
% \cite{goel22TAC}: Competitive control
% \cite{goel20arXiv,goel19PMLR,goel22CDC,goel21PMLR}: Regret-optimal
% with the first three probbably being FI on finite-horizon while
% the last one is measurement-feedback.
% \cite{goel22arXiv}: Best of both worlds
% Papers on OOC: \cite{agarwal19} Also cite (?): Elad
% Hazan. Introduction to online convex optimization. arXiv preprint
% arXiv:1909.05207, 2019. The best of both worlds paper calls
% this a Gradient Perturbation Controller (GPC) using a 
% Disturbance Action Control (DAC) policy.
%
% Don't cite?
% \cite{goel21arXiv,goel22arXivGH}

% -----------------------------------------------------------
\section{Nominal Regret Optimal Control}
\label{sec:nomregret}

\subsection{Notation}
\label{sec:not}

This subsection reviews basic notation regarding vectors, matrices,
signals and systems.  This material can be found in most standard
texts on signals and systems, e.g. \cite{zhou96,dullerud99}.

Let $\R^n$ and $\R^{n\times m}$ denote the sets of real $n\times 1$
vectors and $n\times m$ matrices, respectively. Similarly, $\C^n$ and
$\C^{n\times m}$ denote the sets of complex vectors and matrices of
given dimensions. The superscripts $\top$ and $*$ denote transpose and
complex conjugate (Hermitian) transpose of a matrix.  Moreover, if
$M\in \C^{n\times n}$ then $M^{-\top}$ denotes
$(M^\top)^{-1}=(M^{-1})^\top$.  The Euclidean norm (2-norm) for a
vector $v\in \C^n$ is defined to be
$\| v\|_2:= \sqrt{ v^* v} = \sqrt{ \sum_{i=1}^n |v_i|^2}$.  The
induced 2-norm for a matrix $M \in \C^{n\times m}$ is defined to be:
\begin{align*} 
  \|M\|_{2\to 2}:=\max_{0\ne d\in \C^m} \frac{\|Md\|_2}{\|d\|_2}.
\end{align*}
The induced 2-norm for a matrix $M$ is equal to the maximum singular
value, i.e. $\|M\|_{2\to 2} = \bar{\sigma}(M)$ (Section 2.8 of
\cite{zhou96}). Similarly, $\R^n$ and $\R^{n\times m}$ denote the sets
of $n\times 1$ real vectors and $n\times m$ real matrices,
respectively. The same definitions for the vector 2-norm and
matrix induced 2-norm hold for real vectors and matrices.

The sets of integers and nonnegative integers are denoted by $\Z$ and
$\N$.  Let $v:\Z \to \R^n$ and $w:\Z \to \R^n$ be real, vector-valued
sequences.  Note that we will mainly use two-sided sequences defined
from $t=-\infty$ to $t=+\infty$.  Define the inner product
$\langle v,w \rangle : = \sum_{t=-\infty}^\infty v_t^\top w_t$.  The
set $\ell_2$ is an inner product space with sequences $v$ that satisfy
$\langle v,v\rangle < \infty$. The corresponding norm is
$\|v\|_2 := \sqrt{ \langle v,v \rangle}$. Finally, define the
truncation operator $P_T$ as a mapping from a sequence $v$
to another sequence $w=P_T v$ defined by:
\begin{align*}
  w_t:=\left\{
    \begin{array}{ll}
     v_t  & \mbox{if } t \le T \\ 0 & \mbox{otherwise.}
    \end{array} \right.
\end{align*}

% This is a one-sided $\ell_2$ space, i.e. sequences are defined for
% $t=0$, $1$, $\ldots$.  We will also consider two-sided $\ell_2$ space
% with sequences $v:\Z\to \R^n$, i.e. sequences are defined for
% $t=-\infty$ to $t=+\infty$.

Next, consider a discrete-time, LTI system $G$ with the following
state-space model:
\begin{align}
  \label{eq:Gnotation}
  \begin{split}
  x_{t+1} & = A \, x_t + B\, d_t \\
    e_t & = C \, x_t + D\, d_t,
  \end{split}
\end{align}
where $x_t \in \R^{n_x}$ is the state, $d_t \in \R^{n_d}$ is the
input, and $e_t \in \R^{n_e}$ is the output.  A system $G$ is said to
be causal if $P_T G d = P_T G (P_T d)$, i.e. the output up to time $T$
only depends on the input up to time $T$.  The system is said to be
non-causal if it is not causal, i.e. the output can possibly on future
values of the input.  

The matrix $A\in \R^{n\times n}$ is said to be Schur stable if the
spectral radius is $<1$.  If $A$ is Schur stable then $G$ is a causal,
stable system.  Hence $G$ maps an input $d\in\ell_2$ to output
$e\in \ell_2$ starting from the initial condition $x_{-\infty}=0$.
The induced $\ell_2$-norm for a stable system $G$ is defined to be:
\begin{align*}
  \|G\|_{2\to 2}:=\max_{0\ne d\in \ell_2} \frac{\|G d\|_2}{\|d\|_2}.
\end{align*}
Finally, the transfer function for \eqref{eq:Gnotation}
is $G(z) = C (z I_{n_x} - A)^{-1} B + D$. The $H_\infty$ norm
for a stable system $G$ is:
\begin{align*}
  \|G\|_\infty :=\max_{\theta \in [0,2\pi] }  
   \bar{\sigma}\left( G(e^{j\theta}) \right).
\end{align*}
The induced $\ell_2$-norm for a stable system $G$ is equal to the
$H_\infty$ norm, i.e. $\|G\|_{2\to 2} = \|G\|_\infty$ (Theorem 2.3.2 of
\cite{dahleh94}).

\subsection{Problem Formulation}
\label{sec:nomprob}

Consider the feedback interconnection shown in Figure~\ref{fig:FLPK}.
The interconnection, denoted $F_L(P,K)$, consists of a controller $K$
in feedback around the lower channels of the plant $P$. This is a
standard feedback diagram for optimal control formulations in the
robust control literature \cite{zhou96,dullerud99}.  The plant $P$ is
a discrete-time, linear time-invariant (LTI) system with the following
state-space representation:
\begin{align}
  \label{eq:Pic}
   \bmtx x_{t+1}\\ e_t \\ y_t\emtx =
   \bmtx A & B_d & B_u\\
   C_e & 0 & D_{eu} \\
   C_y & D_{yd} & 0  \\
   \emtx \bmtx x_t \\ d_t \\u_t \emtx,
\end{align}
where $x_t \in \R^{n_x}$ is the state, $d_t \in \R^{n_d}$ is the
disturbance, $e_t \in \R^{n_e}$ is the error, $u_t \in \R^{n_u}$ is
the control input and $y_t \in \R^{n_y}$ is the measured output.  The
plant~\eqref{eq:Pic} assumes zero feedthrough matrices from $u$ to $y$
and from $d$ to $e$, i.e. $D_{yu}=0$ and $D_{ed}=0$.  A standard
loop-shift transformation can be used, under mild technical
conditions, to convert plants with $D_{yu}\ne 0$ and/or $D_{ed} \ne 0$
into the form of Equation~\ref{eq:Pic} (Section 17.2 of
\cite{zhou96}).

\begin{figure}[h]
\centering
\scalebox{0.9}{
\begin{picture}(140,90)(23,20)
 \thicklines
% I/O for P
 \put(75,65){\framebox(40,40){$P$}}
 \put(160,92){$d$}
 \put(155,95){\vector(-1,0){40}}  
 \put(25,92){$e$}
 \put(75,95){\vector(-1,0){40}}  
% I/O for K
 \put(80,25){\framebox(30,30){$K$}}
 \put(43,69){$y$}
 \put(55,75){\line(1,0){20}}  
 \put(55,75){\line(0,-1){35}}  
 \put(55,40){\vector(1,0){25}}  
 \put(141,69){$u$}
 \put(135,40){\line(-1,0){25}}  
 \put(135,40){\line(0,1){35}}  
 \put(135,75){\vector(-1,0){20}}  
\end{picture}
} % End scalebox
\caption{Generic feedback interconnection
$F_L(P,K)$ for synthesis.}
\label{fig:FLPK}
\end{figure}
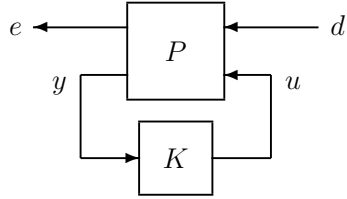

% \footnote{Specifically, let $K$ be a controller
%   designed for \eqref{eq:Pic}. Then $K(I+D_{yu} K)^{-1}$ is an
%   equivalent controller for the plant with $D_{yu} \ne 0$ (assuming
%   $(I+D_{yu}K)^{-1}$ is well-posed). The loopshift for $D_{ed}\ne 0$ is
%   more technical and is specific to optimal synthesis measured with
%   the $H_\infty$ norm. Optimal $H_2$ synthesis typically assumes
%   $D_{ed}=0$ to ensure finite closed-loop cost and hence the
%   loopshift is not needed.}

The goal is to design an output-feedback controller $K$ to stabilize
the plant and ensure the error remains ``small''. The cost achieved by
a controller $K$ on a disturbance $d \in \ell_2$ (assuming
$x_{-\infty}=0$) is:
\begin{align}
  \label{eq:JKd}
  J(K,d):= \| e\|_2^2 = \sum_{t=-\infty}^\infty e_t^\top e_t.
\end{align}
Here the cost is defined using a two-sided disturbance $d\in \ell_2$
defined from $t=-\infty$ to $t=\infty$.  It is common to formulate
optimal control problems using one-sided $\ell_2$ signals starting
from $t=0$ with the initial condition $x_0=0$.  However, a non-causal
controller will be introduced later.  Two-sided signals are used to
avoid nonzero initial conditions arising from this non-causal
controller.

To simplify notation, define $Q:=C_e^\top C_e$, $S:=C_e^\top D_{eu}$
and $R:=D_{eu}^\top D_{eu}$.  The cost can be re-written as:
\begin{align}
  \label{eq:Jqsr}
  J(K,d):= \sum_{t=-\infty}^\infty \bmtx x_t \\ u_t \emtx^\top 
            \bmtx Q & S \\ S^\top & R \emtx \bmtx x_t \\ u_t \emtx.
\end{align}
Equation~\ref{eq:Jqsr} is often used as the starting point for optimal
control problems with a linear quadratic cost. If $(Q,S,R)$ are given
then we can convert to equivalent output matrices $(C_e,D_{eu})$.  For
example, assume $Q\succeq 0$, $R\succ 0$, and $S=0$ are given.  Then
the corresponding error $e$ is obtained with
$C_e := \bsmtx Q^{1/2} \\ 0 \esmtx$ and
$D_{eu} := \bsmtx 0 \\ R^{1/2} \esmtx$.

% For example, suppose $(Q,S,R)$ are given with
% $\bsmtx Q & S \\ S^\top & R \esmtx \succeq 0$ and $R\succ 0$. Then
% $(C_e,D_{eu})$ can be obtained via a matrix square root:
% $\bsmtx Q & S \\ S^\top & R \esmtx = \bsmtx C_e^\top \\ D_{eu}^\top
% \esmtx \bsmtx C_e & D_{eu} \esmtx$.  If $S=0$ then then the
% corresponding error $e$ is obtained with
% $C_e := \bsmtx Q^{1/2} \\ 0 \esmtx$ and
% $D_{eu} := \bsmtx 0 \\ R^{1/2} \esmtx$.

We will use the optimal non-causal controller, denoted $K^o$, as a
baseline for comparison following along the lines of
\cite{goel22TAC,goel20arXiv,goel22CDC,goel21PMLR,sabag21ACC,sabag21arXiv,sabag22arXiv}. The
controller $K^o$ depends directly on past, present, and future values
of $d$.  This controller is described in detail in
Section~\ref{sec:K0} below with an explicit state-space realization
for $K^o$ given in Theorem~\ref{thm:K0}.  In contrast, a causal,
output-feedback controller depends only on past and present values of
the disturbance indirectly via the effect of $d$ on the measurements.
We define the performance of any (causal, output-feedback) controller
$K$ relative to the baseline, non-causal controller $K^o$ as follows:

\begin{defin}
\label{def:nomregret}
Let $\gamma_d \ge 0$ and $\gamma_J \ge 0$ be given.  A controller $K$
achieves $(\gamma_d,\gamma_J)$-regret relative to the optimal
non-causal controller $K^o$ if $F_L(P,K)$ is stable and:
\begin{align}
  \label{eq:nomregret}
  J(K,d) < \gamma_d^2 \, \| d \|_2^2 + \gamma_J^2 \, J(K^o,d)
  \,\,\, \forall d \in \ell_2, \, d\ne 0.
\end{align}
\end{defin}

% This definition of $(\gamma_d,\gamma_J)$-regret includes several
% existing synthesis methods as special cases:

Section~\ref{sec:nomfeas} provides a method to solve the following
$(\gamma_d,\gamma_J)$-regret feasibility problem: Given
$(\gamma_d,\gamma_J)$, find a causal, output-feedback controller $K$
that achieves $(\gamma_d,\gamma_J)$-regret relative to $K^o$ or verify
that this level of regret cannot be achieved. This feasibility problem
includes several existing synthesis methods as special cases:
\begin{itemize}
\item \emph{$H_\infty$ Synthesis}: The $H_\infty$ feasibility problem
  is: Given $\gamma_\infty$, find a controller $K$ such that
  $\|F_L(P,K)\|_\infty<\gamma_\infty$ or verify that this level of
  performance cannot be achieved~\cite{zhou96,dullerud99,dgkf89}.
  There are a variety of existing methods to solve the $H_\infty$
  feasibility problem in discrete-time including the use of: (i)
  Riccati equations~\cite{iglesias91,limebeer89,stoorvogel90}, (ii)
  linear matrix inequalities~\cite{gahinet94}, or (iii) bilinear
  transformations combined with solutions for the continuous-time
  problem~\cite{dgkf89,zhou96}. The objective in $H_\infty$ synthesis
  is to minimize the closed-loop $H_\infty$ norm:
  $\inf_K \|F_L(P,K)\|_\infty$. This optimization can be solved to
  within any desired tolerance using bisection and the solution of the
  $H_\infty$ feasibility problem.  To connect this to regret, recall
  that the $H_\infty$ norm of an LTI system is equal to the induced
  $\ell_2$ norm. Hence the closed-loop with a controller $K$ satisfies
  $\|F_L(P,K)\|_\infty<\gamma_\infty$ if and only if
  $J(K,d)=\|e\|_2^2 < \gamma_\infty^2 \|d\|_2^2$ for all nonzero
  $d\in \ell_2$.  Thus $\|F_L(P,K)\|_\infty<\gamma_\infty$ if and only
  if a controller $K$ achieves $(\gamma_\infty,0)$-regret. Note that
  $(\gamma_\infty,0)$-regret does not depend on the non-causal
  controller $K^o$.

\item \emph{Competitive Ratio Synthesis}: The competitive ratio
  feasibility problem is: Given $\gamma_C$, find a controller $K$ such
  that $\frac{J(K,d)}{J(K^o,d)} < \gamma_C^2$ for all nonzero
  $d\in \ell_2$ or verify that this level of performance cannot be
  achieved~\cite{goel22TAC,sabag22arXiv}. Thus a controller yields a
  competitive ratio of $\gamma_C$ if and only if it achieves
  $(0,\gamma_C)$-regret with respect to $K^o$.  Again, the competitive
  ratio can be minimized to within any desired tolerance using
  bisection and the solution of the $(0,\gamma_C)$-regret feasibility
  problem.

\item \emph{(Additive) Regret-Based Synthesis}: The $\gamma_R$-regret
  feasibility problem is: Given $\gamma_R$, find a controller $K$ such
  that $J(K,d)-J(K^o,d) < \gamma_R^2 \|d\|_2^2$ for all nonzero
  $d\in \ell_2$ or verify that this level of performance cannot be
  achieved~\cite{sabag21ACC,sabag21arXiv,goel20arXiv,goel22CDC}.  Thus
  a controller yields $\gamma_R$-regret if and only if it achieves
  $(\gamma_R,1)$-regret with respect to $K^o$.  Again, the
  $\gamma_R$-regret can be minimized to within any desired tolerance
  using bisection with the $(\gamma_R,1)$-regret
  feasibility problem. The $\gamma_R$-regret is additive in the sense
  that $J(K,d)$ is within an additive factor $\gamma_R^2 \|d\|_2^2$ of
  the optimal non-causal cost.
\end{itemize}

In general we can use the $(\gamma_d,\gamma_J)$-regret feasibility
problem to solve for the Pareto optimal front of values for
$(\gamma_d,\gamma_J)$.  Specifically, the Pareto front can be
characterized by the following optimization parameterized
by $\theta\in [0,1]$:
\begin{align}
  \label{eq:ParetoOpt}
  \begin{split}
    \gamma^*(\theta):= & \min_{K,\gamma} \gamma   \\
    & \mbox{subject to: $K$ achieves
    $\left(\, \sqrt{1-\theta} \cdot \gamma,  \, 
    \sqrt{\theta} \cdot \gamma \,\right)$ regret.}
  \end{split}
\end{align}
The constraint with $\theta=0$ corresponds to the performance bound
$J(K,d) < \gamma^2 \, \| d \|_2^2$.  This yields $H_\infty$ synthesis,
i.e. $\gamma^*(0)=\gamma_\infty$.  Similarly, the constraint with
$\theta=1$ is $J(K,d) < \gamma^2 \, J(K^o,d)$ and this yields
competitive ratio synthesis, i.e. $\gamma^*(1)=\gamma_C$.  These are
extreme endpoints on the Pareto front. Any other value of
$\theta\in [0,1]$ corresponds to the bound
$J(K,d) < \gamma^2 \, \left( (1-\theta) \cdot \| d \|_2^2 + \theta
  \cdot J(K^o,d) \right)$. Thus all other points on the Pareto front
can be viewed as optimizing with respect to a convex combination of
the $H_\infty$ and competitive ratio costs.  The special case of
additive regret synthesis gives the specific point $(\gamma_R,1)$ on
the Pareto front.  This corresponds to the value of $\theta$ such that
$\sqrt{\theta}\cdot \gamma^*(\theta)=1$. (Such a value of $\theta$
exists under mild technical conditions.\footnote{The non-causal
  controller provides a lower bound on the performance of any causal
  controller: $J(K^o,d) \le J(K,d)$ for any causal controller $K$.
  Hence the competitive ratio must satisfy $\gamma_C \ge 1$ which
  further implies
  $\sqrt{\theta} \cdot \gamma^*(\theta) = \gamma_C \ge 1$ for
  $\theta=1$. Moreover, $\gamma^*(\theta)$ is bounded if the plant is
  stabilizable and observable.  Finally, if $\gamma^*(\theta)$ is a
  continuous function then it will cross a value of 1 for some value
  of $\theta \in [0,1]$.})  Finally, we note that
$(\gamma_d,\gamma_J)$-regret is a special case of regret-optimal
control with weights (Section 4.2 of \cite{sabag22arXiv}).

% In general we can use the $(\gamma_d,\gamma_J)$-regret feasibility
% problem to solve for the Pareto optimal front of values for
% $(\gamma_d,\gamma_J)$.  This front contains the $H_\infty$ solution
% $(\gamma_\infty,0)$ and competitive ratio solution $(0,\gamma_C)$ as
% extreme endpoints.  The solution of the additive regret yields one
% point $(\gamma_R,1)$ between these two endpoints.  Finally, we note
% that $(\gamma_d,\gamma_J)$-regret is a special case of 
% regret-optimal control with weights (Section 4.2 of
% \cite{sabag22arXiv}).

% What does the plot look like?  Note that the competitive ratio is
% $\ge 1$.

\subsection{Optimal Non-Causal Controller}
\label{sec:K0}

The optimal non-causal controller is assumed to have full knowledge of
the plant dynamics, plant state and the (past, current and future)
values of the disturbance.  $K^o$ is optimal in the sense that it
minimizes $J(K,d)$ for each $d\in \ell_2$.  A solution for the optimal
non-causal controller is given in Theorem 11.2.1 of
\cite{hassibi99}. The controller is expressed as an operator with
similar results used in \cite{sabag21ACC,sabag21arXiv,sabag22arXiv}.
An explicit state-space model for the finite-horizon, non-causal
controller is constructed in \cite{goel20arXiv,goel22CDC} using
dynamic programming.  These prior results were given for the case
where $S=0$ in the cost function~\eqref{eq:Jqsr}.  The corresponding
infinite-horizon result is given below as Theorem~\ref{thm:K0}
allowing for $S\ne 0$.  An independent proof is given here for
completeness using a completion of the squares argument.

The state-space model for $K^o$ will be constructed using the
stabilizing solution $X$ of a discrete-time algebraic Riccati equation
(DARE).  The next lemma gives sufficient
conditions for the existence of this stabilizing solution of the DARE.

\begin{lemma}
\label{lem:DARE}
Let $(A,B_u,C_e,D_{eu})$ be given and define $Q:=C_e^\top C_e$,
$S:=C_e^\top D_{eu}$, and $R:=D_{eu}^\top D_{eu}$.  

Assume: (i) $R \succ 0$, (ii) $(A,B_u)$ is stabilizable, (iii)
$A-B_uR^{-1}S^\top$ is nonsingular, and (iv)
$\bsmtx A-e^{j\theta} I & B_u \\ C_e & D_{eu} \esmtx$ has full column
rank for all $\theta \in [0,2\pi]$. Then there is a unique stabilizing
solution $X \succeq 0$ such that:
\begin{enumerate}
%\item  $(R+B_u^\top XB_u)$ is nonsingular.
\item $X$ satisfies the following DARE:
\begin{align}
\label{eq:DARE}
  0 & = X  - A^\top XA  - Q \\
\nonumber
    & + (A^\top XB_u+S)\, (R+B_u^\top XB_u)^{-1} \, (A^\top XB_u+S)^\top.
\end{align}
\item The gain
$K_x:=(R+B_u^\top X B_u)^{-1}(A^\top X B_u+S)^\top$
is stabilizing, i.e.  $A-B_uK_x$ is a Schur matrix.
\item  $A-B_uK_x$ is nonsingular.
\end{enumerate}
\end{lemma}
\begin{proof}
  Statements 1)-2) follow from Corollary 21.13 and Theorem 21.7 of
  \cite{zhou96} (after aligning the notation).  Note that the
  stabilizing solution $X$, if it exists, is such that $(R+B_u^\top XB_u)$
  is nonsingular. The matrix inversion
  lemma can be used to show:
  \begin{align*}
    %\label{eq:ABuKx}
    A-B_u K_x =  (I+B_u R^{-1}B_u^\top X)^{-1} \, (A-B_uR^{-1}S^\top).
  \end{align*}
  Statement 3) follows from this equation and the assumption that
  $(A-B_uR^{-1}S^\top)$ is invertible.
\end{proof}

% XXX Moreover, note that (A-BuKx) is nonsingular. Proof: 
% Let Ahat=A-B inv(R) S'.  Then 
%    A-BKx = Ahat - B inv(R+B X B') Bt X Ahat 
%                  = inv(I+B inv(R) B' X) Ahat (by the MIL) 
% Thus nonsingularity of A-BKx follows because we assumed Ahat nonsingular

% XXX Define
%  Ahat = A-B inv(R) D'C = A-B inv(R) S'
%  Qhat = C' (I-D inv(R) D') C = Q-S inv(R) S'
% Then  X>0 if and only if (Qhat,Ahat) has no unobservable stable modes.
% This follows from Thm 21.11 of ZDG.

The special case $S=0$ corresponds to
$C_e := \bsmtx Q^{1/2} \\ 0 \esmtx$ and
$D_{eu} := \bsmtx 0 \\ R^{1/2} \esmtx$. In this case the conditions
(i)-(iv) of Lemma~\ref{lem:DARE} simplify to: (i) $R\succ 0$, (ii)
$(A,B_u)$ stabilizable, (iii) $A$ is nonsingular, and (iv) $(A,Q)$ has
no unobservable modes on the unit circle.  The next theorem constructs
the optimal non-causal controller using the stabilizing solution of the
DARE (again allowing for $S\ne 0$).

\begin{thm}
\label{thm:K0}
Assume $(A,B_u,C_e,D_{eu})$ satisfy conditions $(i)-(iv)$ in
Lemma~\ref{lem:DARE}. Let $X \succeq 0$ be the unique stabilizing
solution of the DARE with corresponding gain $K_x$.

Define a non-causal controller $K^o$ with inputs $(x_t,d_t)$ and
output $u_t^o$ by the following update equations:
\begin{align}
  \label{eq:K0}
  \begin{split}
    v_t & = (A-B_u K_x)^\top ( v_{t+1} + X B_d d_t ), 
    \,\,\, v_\infty=0 \\
    u^o_t & =  -K_x x_t - K_v v_{t+1} - K_d  d_t,
  \end{split}
\end{align}
where 
\begin{align*}
  K_v & := (R+B_u^\top X B_u)^{-1} B_u^\top, \\
  K_d & := (R+B_u^\top X B_u)^{-1} B_u^\top  X B_d.
\end{align*}
Then $J(K^o,d) \le J(K,d)$ for any stabilizing controller $K$ and
disturbance $d\in\ell_2$.
\end{thm}
\begin{proof}
  The optimality of the non-causal controller will be shown via
  completion of the squares. First, define
  $H:=R+B_u^\top X B_u\succ 0$ and express the DARE as:
  \begin{align}
    0=X-A^\top X A - Q + K_x^\top H K_x. 
  \end{align}
  Substitute for $Q$ using the DARE to show:
  \begin{align*}
    \bsmtx Q & S \\ S^\top & R \esmtx = 
    \bsmtx K_x^\top \\ I \esmtx H \bsmtx K_x & I \esmtx
    + \bsmtx X & 0 \\ 0 & 0 \esmtx   
    -\bsmtx A^\top \\ B_u^\top \esmtx X \bsmtx A & B_u \esmtx.
  \end{align*}
  Thus the per-step cost achieved by any controller $K$ is:
  \begin{align*}
    & \bmtx x_t \\ u_t \emtx^\top \bmtx Q & S \\ S^\top & R \emtx 
    \bmtx x_t \\ u_t \emtx    
    =  (u_t+K_x x_t)^\top H (u_t +K_x x_t) \\
    & \,\, + x_t^\top X x_t - (A x_t + B_u u_t)^\top X (Ax_t+B_u u_t).
  \end{align*}
  The per-step cost can be rewritten in terms of the input $u^o$
  generated by the non-causal controller $K^o$:
  \begin{align}
    \label{eq:perstep1}
    \begin{split}
    & \bmtx x_t \\ u_t \emtx^\top \bmtx Q & S \\ S^\top & R \emtx 
    \bmtx x_t \\ u_t \emtx    
    = (u_t-u_t^o)^\top H (u_t - u_t^o) \\
    & - (K_d d_t + K_v v_{t+1})^\top H
        (K_d d_t + K_v v_{t+1}) \\
    & - 2 (K_d d_t + K_v v_{t+1})^\top H (u_t +K_x x_t) \\
    & + x_t^\top X x_t  - (A x_t + B_u u_t)^\top X (Ax_t+B_u u_t). 
  \end{split}
  \end{align}
  The terms on the last two lines can be combined and simplified after
  some algebra. There are two key steps in this simplification.  First
  use the dynamics for $x_{t+1}$ to show:
  \begin{align*}
    x_{t+1}^\top X x_{t+1} & = (Ax_t+B_u u_t)^\top X (A x_t + B_u u_t) \\
        & + 2 d_t^\top B_d^\top X (Ax_t+B_u u_t) + d_t^\top B_d^\top X B_d d_t.
  \end{align*}
  Next, use the dynamics for $x_{t+1}$ and $v_t$ to show:
  {\footnotesize
  \begin{align*}
      & v_t^\top x_t - v_{t+1}^\top x_{t+1} + v_{t+1}^\top B_d d_t  \\
      & \,\,\, = (v_{t+1}+XB_d d_t)^\top (A-B_uK_x) x_t 
      - v_{t+1}^\top (Ax_t+B_u u_t) \\
      & \,\,\, = -(K_d d_t + K_v v_{t+1})^\top H (u_t+K_x x_t)
           + d_t^\top B_d^\top X (Ax_t+B_u u_t).
  \end{align*}}

  \noindent
  Use these two results to simplify the last two lines
  of \eqref{eq:perstep1} thus yielding:
  \begin{align}
    \label{eq:perstep2}
    \begin{split}
    & \bmtx x_t \\ u_t \emtx^\top \bmtx Q & S \\ S^\top & R \emtx 
    \bmtx x_t \\ u_t \emtx    
    = (u_t-u_t^o)^\top H (u_t - u_t^o) \\
    & - (K_d d_t + K_v v_{t+1})^\top H (K_d d_t + K_v v_{t+1}) \\
    & + d_t^\top B_d^\top X B_d d_t +2 v_{t+1}^\top B_d d_t \\
    & + (x_t^\top X x_t - x_{t+1}^\top X x_{t+1})
      + 2(v_t^\top x_t - v_{t+1}^\top x_{t+1} ).
    \end{split}
  \end{align}
  Finally, we obtain the cost by summing from $t=-\infty$ to $\infty$.
  The two terms on the last line of \eqref{eq:perstep2} form
  telescoping sums. These telescoping sums equal zero due to the
  assumption that $x_{-\infty}=0$ and the controller $K$ is
  stabilizing so that $x_t \to 0$ as $t\to \infty$.  Thus the cost is:
  \begin{align}
    \label{eq:Jwithu0}
    J(K,d) 
    & = \sum_{t=-\infty}^\infty \left[ (u_t-u_t^o)^\top H (u_t - u_t^o) \right. \\
    \nonumber
    & - (K_d d_t + K_v v_{t+1})^\top H (K_d d_t + K_v v_{t+1}) \\
    \nonumber
    & \left.  + d_t^\top B_d^\top X B_d d_t +2 v_{t+1}^\top B_d d_t \right]. 
  \end{align}
  The terms on the second and third lines only depend on $d$ and not
  the choice of $u$.  Moreover, $H\succ 0$ and the first term is
  minimized by $u_t=u_t^o$.  Hence the non-causal controller $K^o$
  minimizes the cost.
\end{proof}

The minimal cost achieved by the non-causal
controller is obtained by setting $u_t=u_t^o$ in \eqref{eq:Jwithu0}:
\begin{align*}
  J(K^o,d) 
  & = \sum_{t=-\infty}^\infty 
   \left[ - (K_d d_t + K_v v_{t+1})^\top H (K_d d_t + K_v v_{t+1}) \right. \\
  & \left.  + d_t^\top B_d^\top X B_d d_t +2 v_{t+1}^\top B_d d_t \right]. 
\end{align*}
The non-causal dynamics for $v_t$ in \eqref{eq:K0} are stable when
iterated backward from $v_\infty=0$. If $d\in \ell_2$ then
$v\in \ell_2$ and this cost is finite.

\subsection{Output Feedback Control Design}
\label{sec:nomfeas}

We now return to the $(\gamma_d,\gamma_J)$-regret feasibility problem:
Given $(\gamma_d,\gamma_J)$, find a causal, output-feedback controller
$K$ that achieves $(\gamma_d,\gamma_J)$-regret relative to $K^o$ or
verify that this level of regret cannot be achieved. This section
provides a solution to this feasibility problem. Again, we follow the
basic procedure in
\cite{goel22TAC,goel20arXiv,goel22CDC,goel21PMLR,sabag21ACC,sabag21arXiv,sabag22arXiv}
and use a spectral factorization to reduce the problem to an
equivalent $H_\infty$ feasibility problem.

The regret in Definition~\ref{def:nomregret} involves bounding the
cost $J(K,d)$ achieved by a causal $K$ by the following:
\begin{align*}
\gamma_d^2 \, \| d \|_2^2 + \gamma_J^2 \, J(K^o,d) = 
\sum_{t=-\infty}^\infty   \gamma_d^2 \, d_t^\top d_t 
  + \gamma_J^2 \, (e_t^o)^\top e_t^o.
\end{align*}
Here $e^o$ is the error generated by the closed-loop dynamics with the
non-causal controller $K^o$ in~\eqref{eq:K0}. These closed-loop
dynamics are given by:
\begin{align*}
& x_{t+1}  = \hat{A}_{11} \, x_t - B_u K_v \, v_{t+1} +(B_d- B_u K_d) \, d_t, 
\,\,\,\, x_{-\infty} = 0, \\
& \hat{A}_{11}^\top \, v_{t+1} = v_t - \hat{A}_{11}^\top X B_d \, d_t, 
\,\,\,\, v_\infty = 0,  \\
& e_t^o = (C_e - D_{eu} K_x) \, x_t - D_{eu} K_v \, v_{t+1} - D_{eu} K_d \, d_t,
\end{align*}
where $\hat{A}_{11} := A - B_u K_x$ is a Schur, nonsingular matrix
by Lemma~\ref{lem:DARE}.

We can explicitly solve for $v_{t+1}$ in terms of $(v_t,d_t)$ by
inverting $\hat{A}_{11}^\top$ in the second equation. This allows the
closed-loop dynamics with $K^o$ to be expressed in a simpler form.
Specifically, define an augmented state and error as
$\hat{x}_t := \bsmtx x_t \\ v_t \esmtx$ and
$\hat{e}_t := \bsmtx \gamma_J e_t^o \\ \gamma_d d_t \esmtx$.  The
regret in Definition~\ref{def:nomregret} can be written as:
\begin{align}
\label{eq:RegretBnd}
 J(K,d) < \sum_{t=-\infty}^\infty \hat{e}_t^\top \hat{e}_t = \| \hat{e}\|_2^2
\hspace{0.3in} \forall d \in \ell_2, \, d\ne 0,
\end{align}
where the closed-loop dynamics with $K^o$ have the form
\begin{align}
\label{eq:CLK0}
\begin{split}
\hat{x}_{t+1} & = \hat{A} \, \hat{x}_t + \hat{B} \, d_t \\
  \hat{e}_t & = \hat{C}\, \hat{x}_t + \hat{D} \, d_t,
\end{split}
\end{align}
with the state matrices:
\begin{align*}
&\hat{A}:=\bmtx \hat{A}_{11} & -B_u K_v\hat{A}_{11}^{-\top}
\\ 0 & \hat{A}_{11}^{-\top} \emtx,
&& \hspace{-0.1in} \hat{B}:= \bmtx B_d \\ -X B_d \emtx, \\
& \hat{C}:= \gamma_J \,\bmtx C_e-D_{eu}K_x & -D_{eu}K_v\hat{A}_{11}^{-\top} 
      \\ 0 & 0 \emtx,
&& \hat{D}:= \gamma_d \bmtx 0 \\ I \emtx.
\end{align*}
The relation $K_d = K_vXB_d$ is used to obtain these 
simplified expressions.  Let $\hat{P}$ denote the closed-loop
dynamics in \eqref{eq:CLK0}.

The eigenvalues of $\hat{A}_{11}^{-\top}$ lie outside the unit disk.
These eigenvalues are associated with the stable, non-causal dynamics
of the controller $K^o$. Specifically, an input $d\in \ell_2$
generates $v \in \ell_2$ when iterating the non-causal controller
backward from $v_\infty = 0$. The plant dynamics then generate
$x \in \ell_2$ when iterating forward from $x_{-\infty}=0$.  Moreover,
$v,x \in \ell_2$ imply that $v_t\to 0$ as $t\to -\infty$ and
$x_t\to 0$ as $t\to \infty$.  Hence the augmented state satisfies the
boundary conditions $\hat{x}_{-\infty}=0$ and
$\hat{x}_\infty = 0$.\footnote{The use of two-sided signals ensures
  that the augmented state satisfies zero boundary conditions. This
  would not hold if we used one-sided signals. In particular,
  iterating $K^o$ \eqref{eq:K0} backward from
  $v_\infty=0$ could yield $v_0\ne 0$ and hence
  $\hat{x}_0 = \bsmtx 0 \\ v_0 \esmtx \ne 0$. Such nonzero initial
  conditions create additional technical issues that are avoided
  by using two-sided signals.}

In summary, the system \eqref{eq:CLK0} is stable but with causal and
non-causal dynamics.  A spectral factorization can be used to re-write
$\| \hat{e} \|_2^2$ using only stable, causal dynamics.

% The following spectral factorization result is a corollary of
% Lemma~\ref{lem:DTSF} in Appendix~\ref{app:DTSF}.

\begin{lemma}
\label{lem:RegretSF}
Assume that $(A,B_u,C_e,D_{eu})$ satisfy conditions $(i)-(iv)$ in
Lemma~\ref{lem:DARE} so that the DARE \eqref{eq:DARE} has a
stabilizing solution $X\succeq 0$. This stabilizing solution is used
to define the dynamics $\hat{P}$ from $d_t$ to $\hat{e}_t$ in
\eqref{eq:CLK0}.

If $\gamma_d>0$ and $(\hat{A}_{11}^{-\top},X B_d)$ is stabilizable
then there exists a square $n_d\times n_d$ LTI system $F$ such that:
(i) $\| \hat{e}\|_2^2 = \| F d\|_2^2$ for all $d \in \ell_2$, (ii) $F$
is stable, causal, and invertible, and (iii) $F^{-1}$ is square,
stable and causal.  Moreover, it is possible to construct $F$ with
state dimension $n_x$.
\end{lemma}

It is important to emphasize that $K^o$ is non-causal. Hence the
closed-loop \eqref{eq:CLK0} from $d$ to $\hat{e}$ has both casual and
non-causal dynamics.  The spectral factorization theorem does not
state that $\hat{e}$ can be computed causally. Rather, it states that
the cost $\|\hat{e}\|_2^2$ can be equivalently computed as
$\|F d\|_2^2$ where $F$ is stable and causal. 

The result is a restatement of Lemma~\ref{lem:RegretSFVer2} which is
stated and proved in Appendix~\ref{app:SFforRegret}.  The proof mostly
follows from existing spectral factorization results that are reviewed
in Appendix~\ref{app:DTSF}. The one new aspect of
Lemma~\ref{lem:RegretSF} is to show that the spectral factor $F$ can
be constructed with state dimension $n_x$.  This requires some
additional technical arguments as $\hat{P}$ in \eqref{eq:CLK0} has
state dimension $2n_x$. An explicit state-space construction for
the spectral factorization is given by Lemma~\ref{lem:DTSF}
in Appendix~\ref{app:DTSF} and Lemma~\ref{lem:RegretSFVer2} in 
Appendix~\ref{app:SFforRegret}.

Lemma~\ref{lem:RegretSF} assumes $\gamma_d>0$. This assumption is not
satisfied by competitive control which uses
$(\gamma_d,\gamma_J)=(0,\gamma_C)$. In fact, a spectral factorization
also exists in the case of competitive control (under a weaker set of
assumptions). This requires additional technical details because
$\gamma_d=0$ corresponds to a singular control problem
($\hat{R} = \hat{D}^\top \hat{D} = 0$). These additional details
are omitted.

% \pjs{Say anything
%   more here? The singular control case with $\gamma_d=0$ adds
%   considerable complexity and would be a paper on its own. I definitely
%   want to skip the details for this case.}

We now show that the $(\gamma_d,\gamma_J)$-regret feasibility problem
can be reduced to an $H_\infty$ feasibility problem.  The regret in
Definition~\ref{def:nomregret} can be written as in
\eqref{eq:RegretBnd}: $J(K,d) < \| \hat{e}\|_2^2$ for all nonzero
$d \in \ell_2$.  The left side $J(K,d)$ is equal to the closed-loop
cost $\|e\|_2^2 = \|F_L(P,K) d\|_2^2$ achieved with controller $K$
(Figure~\ref{fig:FLPK}). The right side $\| \hat{e}\|_2^2$ is equal to
$\| F d \|_2^2$ by Lemma~\ref{lem:RegretSF}. The spectral factor $F$
depends on the choice of $(\gamma_d,\gamma_J)$ and is
stable with a stable inverse.  Define $\hat{d}=Fd$ so that
$d=F^{-1}\hat{d}$.  The set of all $d\in \ell_2$ maps 1-to-1 on the
set of all $\hat{d}\in \ell_2$.  Thus we can rewrite the regret bound
as
\begin{align}
  \label{eq:nomregretF}
  \|F_L(P,K) F^{-1} \hat{d} \|_2^2  < \| \hat{d} \|_2^2
  \,\,\, \forall \hat{d} \in \ell_2, \, \hat{d}\ne 0.
\end{align}
The system $F_L(P,K)F^{-1}$ corresponds to the closed-loop $F_L(P,K)$
weighted by the spectral factor inverse $F^{-1}$.  This weighted
closed-loop is shown graphically in Figure~\ref{fig:NomRegretIC}.
Equation~\ref{eq:nomregretF} corresponds to a unit bound on the
closed-loop $H_\infty$ norm, i.e. $\|F_L(P,K) F^{-1}\|_\infty < 1$.
This yields the following main result connecting the regret and
$H_\infty$ feasibility problems.

\begin{thm}
  \label{thm:regnomfeas}
  A controller $K$ achieves $(\gamma_d,\gamma_J)$-regret relative to
  the optimal non-causal controller $K^o$ if and only if
  $F_L(P,K)F^{-1}$ is stable and %satisfies
  $\|F_L(P,K) F^{-1}\|_\infty <1$.
\end{thm}

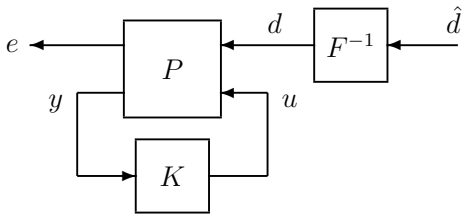
\begin{figure}[h]
\centering
\scalebox{0.9}{
\begin{picture}(200,95)(23,20)
 \thicklines
% I/O for P
 \put(75,65){\framebox(40,40){$P$}}
 \put(135,100){$d$}
 \put(155,95){\vector(-1,0){40}}  
 \put(155,80){\framebox(30,30){$F^{-1}$}}
 \put(210,100){$\hat{d}$}
 \put(215,95){\vector(-1,0){30}}  
 \put(25,92){$e$}
 \put(75,95){\vector(-1,0){40}}  
% I/O for K
 \put(80,25){\framebox(30,30){$K$}}
 \put(43,69){$y$}
 \put(55,75){\line(1,0){20}}  
 \put(55,75){\line(0,-1){35}}  
 \put(55,40){\vector(1,0){25}}  
 \put(141,69){$u$}
 \put(135,40){\line(-1,0){25}}  
 \put(135,40){\line(0,1){35}}  
 \put(135,75){\vector(-1,0){20}}  
\end{picture}
} % End scalebox
\caption{Feedback interconnection $F_L(P,K)F^{-1}$ for regret feasibility.}
\label{fig:NomRegretIC}
\end{figure}

Note that $F^{-1}$ is stable. Hence the controller $K$ stabilizes
$F_L(P,K) F^{-1}$ if and only if $K$ stabilizes $F_L(P,K)$.  There are
well-known results for (sub-optimal) $H_\infty$ synthesis as
summarized in Section~\ref{sec:nomprob}. For example, Theorem 5.2 in
\cite{gahinet94} gives necessary and sufficient conditions to
synthesize a stabilizing controller $K$ such that
$\|F_L(P,K) F^{-1}\|_\infty <1$. These conditions are given in terms
of linear matrix inequalities (LMIs).  No such controller exists if
the LMIs are infeasible.  The only additional assumptions on the
plant $P$ in \eqref{eq:Pic} are that $(A,B_u)$ stabilizable and
$(A,C_y)$ detectable. These are the minimal assumptions required for a
stabilizing controller to exist.  Riccati equation solutions can be
developed either directly in
discrete-time~\cite{iglesias91,limebeer89,stoorvogel90} or mapped, via
the bilinear transform, via solutions for the continuous-time
problem~\cite{dgkf89,zhou96}.  Both of these Riccati 
solutions require additional technical assumptions on the plant
beyond stabilizability and detectability of $(A,B_u,C_y)$.

% In summary, we can construct a controller that achieves
% $(\gamma_d,\gamma_J)$-regret relative to the non-causal controller
% $K^o$ (or determine that one does exist) by the following steps: (i)
% construct the spectral factor $F$ corresponding to the given
% $(\gamma_d,\gamma_J)$, and (ii) use existing $H_\infty$ synthesis
% methods to find a stabilizing controller $K $ such that
% $\|F_L(P,K) F^{-1}\|_\infty <1$ (or determine that one does not
% exist).  The controller $K$ is dynamic, in general, and takes
% measurements $y$ to produce control commands $u$.  The weighted plant
% used for synthesis has state dimension $2n_x$ because both $P$ and
% $F^{-1}$ have state dimension $n_x$.  It is known that if the
% $H_\infty$ problem is feasible then the dynamic controller $K$ can
% constructed with the same state-dimension as the weighted plant,
% i.e. $K$ can be constructed with order $2n_x$.

In summary, we can construct a controller that achieves
$(\gamma_d,\gamma_J)$-regret relative to the non-causal controller
$K^o$ (or determine that one does exist) by the following steps: 

\begin{enumerate}
\item Construct the optimal non-causal controller as
in Theorem~\ref{thm:K0}.

\item Construct the spectral factor $F$ for the given
$(\gamma_d,\gamma_J)$ based on Lemma~\ref{lem:RegretSF} and
Appendix~\ref{app:DTSF}, and

\item Use existing $H_\infty$ synthesis methods~\cite{doyle89} to find
  a stabilizing controller $K $ such that
  $\|F_L(P,K) F^{-1}\|_\infty <1$ (or determine that one does not
  exist).
\end{enumerate}

\noindent
The resulting controller $K$ is dynamic, in general, and takes
measurements $y$ to produce control commands $u$.  The weighted plant
used for synthesis has state dimension $2n_x$ because both $P$ and
$F^{-1}$ have state dimension $n_x$.  It is known that if the
$H_\infty$ problem is feasible then the dynamic controller $K$ can
constructed with the same state-dimension as the weighted plant,
i.e. $K$ can be constructed with order $2n_x$.

Finally, we comment on the special case of full-information control.
This corresponds to the case where the controller can directly measure
the disturbance and plant state: $y_t = \bsmtx x_t \\ d_t \esmtx$.
Full-information regret-based control has been studied extensively, e.g.
\cite{hassibi99,sabag21ACC,sabag21arXiv,goel22TAC,goel21arXiv}.\footnote{These
  previous works used the per-step cost
  $x_t^\top Q x_t + u_t^\top R x_t$.  This corresponds to the case
  $S=0$ in the cost function~\eqref{eq:Jqsr}.
  Theorem~\ref{thm:regnomfeas} can be applied for more general
  per-step costs with $S\ne 0$.}  The controller directly measures
$d_t$ in the full-information problem.  The measurement of $d_t$ can
be used by the controller to perfectly reconstruct $\hat{d}_t$ via
$\hat{d}=Fd$. Let $\hat{x}_t$ denote the state of $F$ evolving with
input $d$. This is equal to the state of $F^{-1}$ evolving with input
$\hat{d}$. Thus the full-information $H_\infty$ synthesis problem can
be equivalently solved using
$\bsmtx x_t^\top & \hat{x}_t^\top & \hat{d}_t^\top \esmtx^\top$ as the
measurement.  Full-information $H_\infty$ synthesis is solved by a
static (memoryless) controller:
\begin{align}
  \label{eq:FIregretK}
  u_t = -K_x^{FI}  x_t - K_{\hat{x}}^{FI} \hat{x}_t - K_{\hat{d}}^{FI} \hat{d}_t.
\end{align}
Details on the construction of this controller for discrete-time case
can be found in Chapter 9 of \cite{stoorvogel90} or Appendix B of
\cite{green12}.  This controller can be implemented, in real-time, by
using the measurement of $d$ to compute $\hat{d}=Fd$ and the state
$\hat{x}_t$. Thus the full-information, $(\gamma_d,\gamma_J)$-regret
problem is solved by a controller that combines the dynamics
$F$ with the static update in \eqref{eq:FIregretK}.  This is
a dynamic controller, due to the dependence on $F$, with state
dimension $n_x$.

% Thus the full-information, $(\gamma_d,\gamma_J)$-regret problem
% requires a dynamic controller with state dimension $n_x$.

% -----------------------------------------------------------
\section{Robust Regret Optimal Control}
\label{sec:robregret}

\subsection{Problem Formulation}
\label{sec:robprob}

Consider the feedback interconnection shown in
Figure~\ref{fig:FLPKUnc}.  This interconnection, denoted
$CL(P,K,\Delta)$, includes an uncertainty $\Delta$ and controller $K$
wrapped around the upper and lower channels of the plant $P$,
respectively.  This is a standard feedback diagram in the robust
control literature for the case of uncertain systems, e.g. see Chapter
11 of \cite{zhou96} or Chapter 8 of \cite{dullerud99}.  The plant $P$
is a discrete-time, LTI system with additional input/output channels
$(w,v)$ to incorporate the effect of the uncertainty:
\begin{align}
  \label{eq:Punc}
   \bmtx x_{t+1}\\ v_t \\ e_t \\ y_t\emtx =
   \bmtx A & B_w & B_d & B_u\\
   C_v & D_{vw} & D_{vd} & D_{vu} \\
   C_e & 0 & 0 & D_{eu} \\
   C_y & D_{yw} & D_{yd} & 0  \\
   \emtx \bmtx x_t \\ w_t \\ d_t \\u_t \emtx,
\end{align}
where $w_t \in \R^{n_w}$ and $v_t \in \R^{n_v}$ are the input/output
signals associated with the uncertainty.  The other signals have
similar interpretations as given for the nominal case in
Section~\ref{sec:nomprob}.

% Dyw nonzero allows FU(P,Delta) to have a direct feedthrough from
% u to y. (It goes from u-->v via Dvu then v-->w via Delta.D then
% w-->uy via Dyw.) However, this is not an issue because the synthesis
% in DK-syn is performed direct on a scaled version of P. The scaled
% plant P (with D-scales) does not include Delta and hence does not
% have a direct feedthrough from u to y.

\begin{figure}[h]
\centering
\scalebox{0.85}{
\begin{picture}(140,126)(23,22)
 \thicklines
% I/O for P
 \put(75,65){\framebox(40,40){$P$}}
 \put(160,82){$d$}
 \put(155,85){\vector(-1,0){40}}  
 \put(25,82){$e$}
 \put(75,85){\vector(-1,0){40}}  
% I/O for K
 \put(80,22){\framebox(30,30){$K$}}
 \put(43,67){$y$}
 \put(55,72){\line(1,0){20}}  
 \put(55,72){\line(0,-1){35}}  
 \put(55,37){\vector(1,0){25}}  
 \put(141,67){$u$}
 \put(135,37){\line(-1,0){25}}  
 \put(135,37){\line(0,1){35}}  
 \put(135,72){\vector(-1,0){20}}  
% I/O for Delta
 \put(80,118){\framebox(30,30){$\Delta$}}
 \put(43,98){$v$}
 \put(55,98){\line(1,0){20}}  
 \put(55,98){\line(0,1){35}}  
 \put(55,133){\vector(1,0){25}}  
 \put(141,98){$w$}
 \put(135,133){\line(-1,0){25}}  
 \put(135,133){\line(0,-1){35}}  
 \put(135,98){\vector(-1,0){20}}  
\end{picture}
} % End scalebox
\caption{Feedback interconnection $CL(P,K,\Delta)$ for robust synthesis.} 
\label{fig:FLPKUnc}
\end{figure}
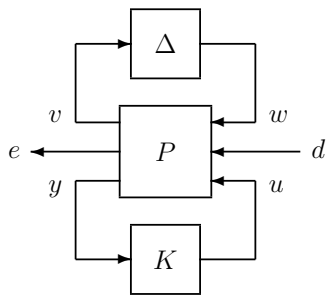

Let $\mathbf{\Delta}$ denote the set of $n_w\times n_v$, discrete-time
LTI systems that are causal, stable, and have $\|\Delta\|_\infty \le 1$.
Thus each $\Delta \in \mathbf{\Delta}$ has a state-space representation:
\begin{align}
   \bmtx x^\Delta_{t+1}\\ w_t  \emtx =
   \bmtx A_\Delta & B_\Delta \\
      C_\Delta & D_\Delta \emtx
     \bmtx x^\Delta_t \\ v_t \emtx.
\end{align}
We only assume that each $\Delta \in \mathbf{\Delta}$ has a
finite-dimensional state but the state dimension is arbitrary. This is
referred to as unstructured uncertainty in the robust control
literature \cite{zhou96}. The results in this paper can be extended to
LTI uncertainty with infinite state dimension but this requires
additional technical machinery.

% Our only assumption on the uncertainty is that it lies within this
% set, i.e. $\Delta \in \mathbf{\Delta}$. We only assume that $\Delta$
% has a finite-dimensional state-space representation:

Let $F_U(P,\Delta)$ denote the system obtained by closing
$\Delta \in \mathbf{\Delta}$ around the top channels of
$P$.\footnote{The system $F_L( F_U(P,\Delta), K)$ corresponds to
  closing a controller $K$ around the bottom channels of
  $F_U(P,\Delta)$.  Thus $F_L( F_U(P,\Delta), K)$ is equivalent to the
  shorter notation $CL(P,K,\Delta)$.}  The ``nominal'' (or best-guess)
model is obtained with $\Delta = 0$. The results in
Section~\ref{sec:nomregret} can be used to construct a controller that
achieves $(\gamma_d,\gamma_J)$-regret on the nominal model $F_U(P,0)$.
In this section we focus on designing a single controller that
achieves a certain level of regret for all possible models in
the set
$\mathcal{M}:=\{ F_U(P,\Delta) \, : \, \Delta \in \mathbf{\Delta} \}$.

We first need to concretely define several notions used in this robust
design.  The closed-loop system $CL(P,K,\Delta)$ in
Figure~\ref{fig:FLPKUnc} is said to be well-posed if the dynamics have
a unique solution for any disturbance $d \in \ell_2$ starting from
zero initial conditions for $P$ and $\Delta$, i.e. $x_{-\infty}=0$ and
$x^\Delta_{-\infty}=0$. In this case, the cost achieved by a
controller $K$ with a disturbance $d \in \ell_2$ and uncertainty
$\Delta\in\mathbf{\Delta}$ is:
\begin{align}
  \label{eq:JKdDelta}
  J(K,d,\Delta):= \| e\|_2^2 = \sum_{t=-\infty}^\infty e_t^\top e_t.
\end{align}
Our baseline for performance will be the optimal non-causal controller
$K^o$ designed for the nominal model $F_U(P,0)$. This baseline
controller can be constructed from the results in
Section~\ref{sec:K0}.  The cost achieved by $K^o$ on the nominal model
with disturbance $d$ is $J(K^o,d,0)$.  We define the regret relative
to this baseline cost.

% An output feedback controller $K$ achieves nominal
% stability if it stablizes the nominal model $F_U(P,0)$.  The
% controller $K$ achieves robust stability if it stabilizes
% $F_U(P,\Delta)$ for all $\Delta \in \mathbf{\Delta}$.

\begin{defin}
\label{def:robregret}
Let $\gamma_d \ge 0$ and $\gamma_J \ge 0$ be given.  A controller $K$
achieves robust $(\gamma_d,\gamma_J)$-regret relative to the optimal
non-causal controller $K^o$ (designed on the nominal model) if
$CL(P,K,\Delta)$ is well-posed and stable for all
$\Delta \in \mathbf{\Delta}$, and:
\begin{align}
  \label{eq:robregret}
  \begin{split}
  J(K,d,\Delta) < \gamma_d^2 \, \| d \|_2^2 + \gamma_J^2 \, J(K^o,d,0)
  \\ \forall d \in \ell_2, \, d\ne 0 \,\,
  \mbox{ and } \,\, \forall \Delta \in \mathbf{\Delta}.
  \end{split}
\end{align}
\end{defin}

Section~\ref{sec:robfeas} provides a method to solve the following
robust $(\gamma_d,\gamma_J)$-regret feasibility problem: Given
$(\gamma_d,\gamma_J)$, find a causal, output-feedback controller $K$
that achieves robust $(\gamma_d,\gamma_J)$-regret relative to $K^o$ or
verify that this level of robust regret cannot be achieved. Robust
$(\gamma_d,0)$-regret is a special case of $\mu$-synthesis feasibility
problem \cite{doyle85,doyle87,balas94,lind94,packard93asme}. We have
restricted our attention to unstructured uncertainty
$\Delta \in\mathbf{\Delta}$. More general block-structured uncertainty
is considered in the standard $\mu$ synthesis problem
\cite{safonov80,doyle82,packard93,zhou96,dullerud99}. The synthesis
method in Section~\ref{sec:robfeas} can be adapted to these more
general block-structured uncertainties. We will comment later on the
changes required for more general block structures.

% XXX PJS: Comment on IQCs and IQC-synthesis?

\subsection{Output Feedback Control Design}
\label{sec:robfeas}

We first show that the robust $(\gamma_d,\gamma_J)$-regret feasibility
condition can be reduced to a robust $H_\infty$ condition.  The bound
on the right side of the robust regret condition \eqref{eq:robregret}
does not depend on the uncertainty.  Following the same arguments as
in Section~\ref{sec:nomfeas}, this bound can be written as
$\| \hat{e} \|_2^2$ where $\hat{e}$ is the output of the system
\eqref{eq:CLK0}. Moreover, $\| \hat{e}\|_2^2$ is equal to
$\| F d \|_2^2$ by Lemma~\ref{lem:RegretSF} where $F$ is a spectral
factor that depends on the choice of $(\gamma_d,\gamma_J)$.  Thus the
robust $(\gamma_d,\gamma_J)$-regret in Definition~\ref{def:robregret}
can be written as:
\begin{align*}
  J(K,d,\Delta) < \| F d \|_2^2
  \hspace{0.2in} \forall d \in \ell_2, \, d\ne 0 \,\,
  \mbox{ and } \,\, \forall \Delta \in \mathbf{\Delta}.
\end{align*}
The left side $J(K,d,\Delta)$ is equal to the closed-loop cost
$\|e\|_2^2 = \|CL(P,K,\Delta) \, d\|_2^2$ achieved with controller $K$
and uncertainty $\Delta \in \mathbf{\Delta}$.
(Figure~\ref{fig:FLPKUnc}).  The spectral factor $F$ is stable with a
stable inverse.  Define $\hat{d}=Fd$ so that $d=F^{-1}\hat{d}$.  The
set of all $d\in \ell_2$ maps 1-to-1 on the set of all $\hat{d}\in \ell_2$.
Thus we can rewrite the robust regret bound as
\begin{align}
  \label{eq:robregretF}
  \begin{split}
    \|CL(P,K,\Delta) F^{-1} \hat{d} \|_2^2  < \| \hat{d} \|_2^2 \\
    \hspace{0.2in} \forall \hat{d} \in \ell_2, \, \hat{d}\ne 0 \,\,
    \mbox{ and } \,\, \forall \Delta \in \mathbf{\Delta}.
  \end{split}
\end{align}
The system $CL(P,K,\Delta)F^{-1}$ corresponds to the closed-loop
$CL(P,K,\Delta)$ with the disturbance channel weighted by $F^{-1}$.
Thus \eqref{eq:robregretF} corresponds to a robust performance
condition on $CL(P,K,\Delta)F^{-1}$ that must hold for all
uncertainties $\Delta \in \mathbf{\Delta}$. This yields the following
main result connecting the robust regret to the robust $H_\infty$
condition.

\begin{thm}
  \label{thm:regrobfeas}
  A controller $K$ achieves robust $(\gamma_d,\gamma_J)$-regret
  relative to the optimal non-causal controller $K^o$ if and only if
  $CL(P,K,\Delta) F^{-1}$ is well-posed, stable and %satisfies 
  $\|CL(P,K,\Delta) F^{-1}\|_\infty <1$ for all $\Delta \in \mathbf{\Delta}$.
\end{thm}

The robust performance condition in Theorem~\ref{thm:regrobfeas}
depends on the uncertainty.  We next show that this is equivalent to a
condition that does not depend on the uncertainty.  Consider the
feedback interconnection in Figure~\ref{fig:RobRegIC}.  This
corresponds to the closed-loop $CL(P,K,\Delta)$ with the uncertainty
channels ``open'' and with the disturbance channel weighted by
$F^{-1}$. The system from inputs $\bsmtx w \\ \hat{d} \esmtx$ to
outputs $\bsmtx v \\ e \esmtx$ is denoted by:
\begin{align}
  \label{eq:Mdef}
  M:=F_L(P,K) \bmtx I_{n_w} & 0 \\ 0 & F^{-1}\emtx.
\end{align}
$M$ is constructed from $(P,K,F^{-1})$ but this dependence is
suppressed to simplify the notation.  Note that
$CL(P,K,\Delta) F^{-1} = F_U(M,\Delta)$.

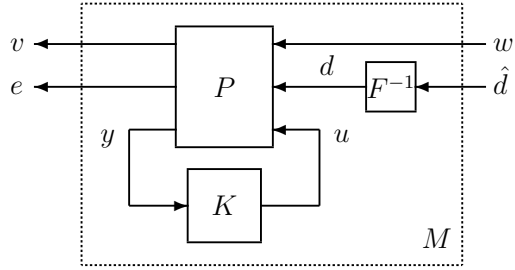
\begin{figure}[h]
\centering
\scalebox{0.9}{
\begin{picture}(210,110)(5,15)
 \thicklines
 % I/O for P
 \put(75,65){\framebox(40,50){$P$}}
 \put(135,95){$d$}
 \put(155,90){\vector(-1,0){40}}  
 \put(155,80){\framebox(20,20){$F^{-1}$}}
 \put(208,87){$\hat{d}$}
 \put(205,90){\vector(-1,0){30}}  
 \put(5,87){$e$}
 \put(75,90){\vector(-1,0){60}}  
% I/O for K
 \put(80,25){\framebox(30,30){$K$}}
 \put(43,67){$y$}
 \put(55,72){\line(1,0){20}}  
 \put(55,72){\line(0,-1){32}}  
 \put(55,40){\vector(1,0){25}}  
 \put(141,67){$u$}
 \put(135,40){\line(-1,0){25}}  
 \put(135,40){\line(0,1){32}}  
 \put(135,72){\vector(-1,0){20}}  
% I/O for Delta
 \put(208,105){$w$}
 \put(205,108){\vector(-1,0){90}}  
 \put(5,105){$v$}
 \put(75,108){\vector(-1,0){60}}  
% Dashed box for M
 \put(35,15){\dashbox(160,110)}
 \put(178,22){$M$}
\end{picture}
} % End scalebox
\caption{Feedback interconnection for robust regret feasibility.}
\label{fig:RobRegIC}
\end{figure}

The next theorem uses $M$ to state a necessary and sufficient
condition for $K$ to achieve robust $(\gamma_d,\gamma_J)$-regret.
This is a special case of results for the structured singular value,
also known as $\mu$ (see \cite{packard93} or Chapter 11 of \cite{zhou96}).
Technical details are given in Appendix~\ref{app:RP}.

\begin{thm}
  \label{thm:MDregbnd}
  A controller $K$ achieves robust $(\gamma_d,\gamma_J)$-regret
  relative to the optimal non-causal controller $K^o$ if and only if
  (i) $M$ is stable, and (ii) there exists $D:[0,2\pi]\to (0,\infty)$
  such that
  \begin{align}
    \nonumber
    & \bar{\sigma}\left( \bsmtx D(\theta) \cdot I_{n_v} & 0 \\ 0 & I_{n_e} \esmtx
    M(e^{j\theta}) \bsmtx D(\theta)^{-1} \cdot I_{n_w} & 0 \\ 0 & I_{n_d} \esmtx
   \right) < 1 \\ 
    \label{eq:MDregbnd}
     & \hspace{2in} \forall \theta \in [0,2\pi].
    \end{align}
\end{thm}
\begin{proof}
  By Theorem~\ref{thm:regrobfeas}, a controller $K$ achieves robust
  $(\gamma_d,\gamma_J)$-regret relative to $K^o$ if and only if
  $F_U(M,\Delta)$ is well-posed, stable and %satisfies
  $\|F_U(M,\Delta) \|_\infty <1$ for all $\Delta \in \mathbf{\Delta}$.
  It follows from Lemma~\ref{lem:SystemSSV} in Appendix~\ref{app:RP}
  that $F_U(M,\Delta)$ satisfies these conditions for all
  $\Delta \in \mathbf{\Delta}$ if and only if (i) $M$ is stable, and
  (ii) there exists $D$ satisfying \eqref{eq:MDregbnd}.
\end{proof}

% \footnote{Suppose $D$ is a SISO, LTI, stable system with $D^{-1}$
%   stable. The system frequency response has magnitude
%   $m(\theta):=|D(e^{j\theta})|>0$ and phase
%   $\phi(\theta):=\angle D(e^{j\theta})$. The maximum singular value is
%   unaffected by the phase. In other words, \eqref{eq:MDregbnd} holds
%   with $m:[0,2\pi]\to (0,\infty)$ if and only if it holds with
%   $m(\theta)e^{j\phi(\theta)}$.}

% \footnote{The use of stable, minimum phase scalings is a minor
%   restriction.  Specifically, assume \eqref{eq:MDregbnd} holds with
%   for some scaling $D:[0,2\pi]\to (0,\infty)$. Then
%   \eqref{eq:MDregbnd} also holds with the scaling
%   $D(e^{j\theta})e^{j\phi(\theta)}$ where $\phi$ is any additional
%   phase. This follows from the unitary invariance of the maximum
%   singular value.  We can select the phase so that
%   $D(e^{j\theta})e^{j\phi(\theta)}$ corresponds to the frequency
%   response of a stable, minimum phase system, e.g. using the complex
%   cepstrum method. This frequency response can be fit arbitrarily
%   closely by a stable, minimum phase system of sufficiently high
%   order.  In summary, if \eqref{eq:MDregbnd} holds with a scaling
%   $D:[0,2\pi]\to (0,\infty)$ then it holds with a SISO, LTI, system
%   that is stable and minimum phase.  }

The robust $(\gamma_d,\gamma_J)$-regret feasibility problem requires
the construction of a controller $K$ and scaling $D$ that satisfy
\eqref{eq:MDregbnd} (where the dependence of $M$ on $K$ has been
suppressed). Unfortunately \eqref{eq:MDregbnd} is, in general, a
non-convex constraint on $(D,K)$.  A common, pragmatic approach is to
perform a coordinate-wise search known as DK-iteration or
$\mu$-synthesis \cite{doyle85,doyle87,balas94,lind94,packard93asme}.
The $D$ scaling is restricted to be a SISO, LTI, stable system with
$D^{-1}$ stable. In this case,
Equation~\ref{eq:MDregbnd} can be written explicitly in terms of
$(P,K,F^{-1},D)$ as follows:
\begin{align}
  \label{eq:DMDinv}
  \left\|
   \bsmtx D \cdot I_{n_v} & 0 \\ 0 & I_{n_e} \esmtx
    F_L(P,K) \bsmtx D^{-1} \cdot I_{n_w} & 0 \\ 0 & F^{-1} \esmtx
  \right\|_\infty < 1.
\end{align}
The DK-iteration corresponds to alternately minimizing the $H_\infty$
norm in \eqref{eq:DMDinv} over $K$ with $D$ held fixed and vice
versa. The $K$-step can be solved as a standard $H_\infty$ synthesis
problem. The $D$-step requires additional details
\cite{doyle85,doyle87,balas94,lind94,packard93asme} but efficient
algorithms exist to compute sub-optimal scalings.  The algorithm is
terminated if the $H_\infty$ norm goes below 1.  If the algorithm
terminates for this reason then we have a controller that achieves
the robust regret bound.  Finally, note that
Theorem~\ref{thm:MDregbnd} gives a necessary and sufficient condition
to achieve robust $(\gamma_d,\gamma_J)$-regret. However, the
DK-iteration need not find the globally optimal pair $(D,K)$. Hence,
successful termination of this algorithm is only sufficient for
achieving robust regret.

\section{Examples}
\label{sec:example}

% This section presents two examples to illustrate the proposed
% synthesis method for robust, regret optimal control. The code
% to reproduce the results is available at:

% \url{https://github.com/jliu879/Robust_Regret_Optimal_Control}

This section presents three examples to illustrate the proposed
synthesis method for robust, regret optimal control.  Code to
reproduce the results for all examples is available at:

\url{https://github.com/jliu879/Robust_Regret_Optimal_Control}

\subsection{Simple SISO Design}

This section presents a simple, SISO design example. This
example is useful to gain some additional insights about the
performance of the various nominal and robust regret controllers
discussed in this paper.  Consider the classical feedback diagram in
Figure~\ref{fig:simpleEx} with the (continuous-time) plant
$G(s)=\frac{15}{s+5.6}$. The actuator dynamics $A(s)$ are
assumed to be uncertain:
\begin{align}
  \label{eq:SimpleA}
  A(s) = A_0(s) \cdot (1+W_{unc}(s)\, \Delta(s) ),
\end{align}
where $W_{unc}(s) = \frac{3s+4.62}{s+23.1}$ and $\Delta(s)$ is a
stable, LTI system satisfying $\|\Delta\|_\infty \le 1$.  The weight
$W_{unc}(s)$ has DC gain of 0.2, $|W_{unc}(j\infty)|=3$, and a zero
near $-1.5$rad/sec.  This represents a relatively small uncertainty
(20\%) uncertainty at low frequencies with increasing uncertainty
above $1.5$rad/sec. The continuous-time plant and actuator
models are discretized using zero-order hold with a sample time
$T_s = 0.001$ sec.\footnote{$A_0(s)$ and $W_{unc}(s)$ are discretized
  with sample time $T_s$. The normalized uncertainty is replaced by a
  discrete-time, stable LTI system $\Delta(z)$ with sample time $T_s$
  satisfying $\|\Delta\|_\infty \le 1$.}

\begin{figure}[h]
\centering
\scalebox{0.9}{
\begin{picture}(380,110)(0,-50)
 \thicklines
 % Classical Diagram
 \put(0,0){\vector(1,0){40}}
 \put(10,3){$d$}
 \put(40,-20){\framebox(40,40){$W_d(s)$}}
 \put(80,0){\vector(1,0){27}}
 \put(110,0){\circle{6,0}}
 \put(113,0){\vector(1,0){37}}
 \put(150,-20){\framebox(40,40){$K(s)$}}
 \put(190,0){\vector(1,0){40}}
 \put(230,-20){\framebox(40,40){$A(s)$}}
 \put(270,0){\vector(1,0){30}}
 \put(300,-20){\framebox(40,40){$G(s)$}}
 \put(340,0){\vector(1,0){40}}
 %\put(370,3){$y$}
 \put(360,0){\line(0,-1){50}}
 \put(360,-50){\line(-1,0){250}}
 \put(110,-50){\vector(0,1){47}}
 \put(115,-15){\line(1,0){8}}
 % Add We weight
 \put(125,0){\line(0,1){50}}
 \put(125,50){\vector(1,0){25}}
 \put(150,35){\framebox(30,30){$W_e(s)$}}
 \put(180,50){\vector(1,0){20}}
 \put(190,57){$e_1$}
 % Add Wu weight
 \put(210,0){\line(0,1){50}}
 \put(210,50){\vector(1,0){25}}
 \put(235,35){\framebox(30,30){$W_u(s)$}}
 \put(265,50){\vector(1,0){20}}
 \put(275,57){$e_2$}
\end{picture}
} % End scalebox
\caption{Feedback interconnection for SISO design example.}
\label{fig:simpleEx}
\end{figure}
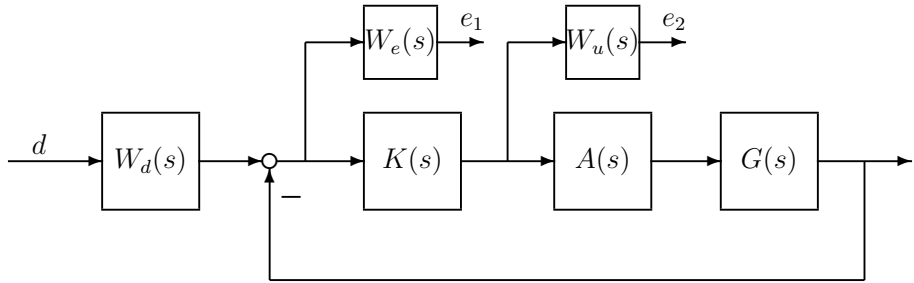

Figure~\ref{fig:simpleEx} also includes weights to describe the
performance objectives.  There is one generalized disturbance $d$
that accounts for the effect of the reference command.
The two generalized errors $e_1$ and
$e_2$ represent the competing objectives between reference tracking
and control effort.  The corresponding weights are:
\begin{align}
  W_d(s) & = \frac{8}{s+8}, \\
  W_e(s) & = \frac{0.5 s + 6.93}{s+0.0693}, \\
  W_u(s) & = \frac{1000s+ 804}{s+8040}.
\end{align}
The weight $W_d(s)$ represents reference commands with dominant low
frequency content below $8$\,rad/sec.  The weight $W_e$ has a DC gain
of 10, crosses a gain of one at 8\,rad/sec, and transitions to a gain
of 0.5 at high frequencies. This weight emphasizes the tracking error
at low frequencies.  Conversely, the weight $W_u(s)$ has a DC gain of
0.1, crosses a gain of one at 8\,rad/sec, and transitions to a gain of
1000 at high frequencies. This weight penalizes high frequency control
effort. These weights are discretized using a zero-order with a sample
time $T_s = 0.001$ sec.  The design problem can expressed in the form
of the robust synthesis interconnection $CL(P,K,\Delta)$ shown in
Figure~\ref{fig:FLPKUnc}.  

% The design model $P$ in Figure~\ref{fig:FLPKUnc} depends on the
% nominal plant and actuator dynamics $(G(s),A_0(s))$, uncertainty
% weight $W_{unc}(s)$, and performance weights $(W_d(s),W_e(s),W_u(s))$.

We first consider the properties of the optimal non-causal controller
$K^o$ constructed in Section~\ref{sec:K0}.  To simplify notation, let
$T^o$ denote the (non-causal) closed-loop system from $d$ to $e$
obtained with $K^o$. A given disturbance $d\in \ell_2$ generates an
error $e$ with resulting cost $J(K^o,d)= \| e \|_2^2$. Let $\hat{d}$
and $\hat{e}$ denote the Fourier Transforms of $d$ and $e$,
respectively, so that
$\hat{e}(j\omega) = T^o(j\omega) \hat{d}(j\omega)$.  It follows from
the Plancheral (Parseval's)  Theorem (Section 3.3 of \cite{dullerud99}) that the
cost can be equivalently computed in the frequency domain:
\begin{align}
    J(K^o,d) 
    & = \frac{1}{2\pi} \int_{-\infty}^{+\infty} 
         \hat{e}(j\omega)^* \,  \hat{e}(j\omega) \, d\omega 
      = \frac{1}{2\pi} \int_{-\infty}^{+\infty} 
        \left\| T^o(j\omega) \right\|_2^2
        \,\, |\hat{d}(j\omega)|^2 \, d\omega.
\end{align}
The simplification in the right-most expression relies on the fact
that $\hat{d}$ is scalar so that $T^o(j\omega)$ is a $2\times 1$
column vector. Figure~\ref{fig:SimpleEx_NoncausalCost} shows the
response $\| T^o(j\omega)\|_2$ vs. $\omega$.  The gain is small at
both very low and very high frequencies. This indicates that the
optimal non-causal controller achieves a small cost for disturbances
(reference commands) with energy predominantly at low and high
frequencies.  However, the closed-loop response has a large gain in
the range of 1 to 10 rad/sec. Thus the optimal non-causal controller
will have (relatively) poor performance for disturbances with energy
predominantly in these mid-frequencies.

\begin{figure}[h!]
  \centering
  \includegraphics[width=0.45\textwidth]{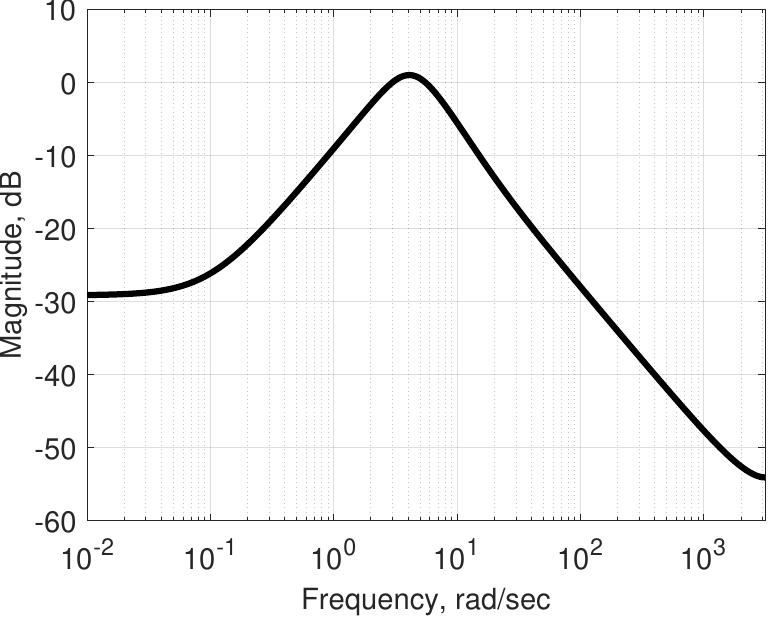}
  \caption{Closed-loop response with optimal non-causal controller:
    $\| T^o(j\omega)\|_2$ vs. $\omega$.}
\label{fig:SimpleEx_NoncausalCost}
\end{figure}

Next we consider the nominal regret, as defined in
Section~\ref{sec:nomregret}, with $\Delta=0$. The nominal regret
includes three special optimal controllers (values within numerical
tolerances): (a) $H_\infty$: $(\gamma_d,\gamma_J) = (1.82,0)$, (b)
competitive ratio: $(\gamma_d,\gamma_J) = (0,3.29)$, and (c) additive
regret: $(\gamma_d,\gamma_J) =
(1.63,1)$. Figure~\ref{fig:SimpleEx_Pareto} shows these three points
along with the Pareto optimal curve for feasible values
$(\gamma_d,\gamma_J)$ of nominal regret (blue dashed curve).
Equation~\ref{eq:ParetoOpt} in Section~\ref{sec:nomprob} characterized
the Pareto front using a parameterized optimization.  Here we use a
different numerical implementation and generate the curve by: (i)
selecting 20 values of $\gamma_d$ in the range of
$[0.001,\, 0.999] \times \gamma_\infty$ where $\gamma_\infty = 1.82$,
and (ii) bisecting to find the minimal value of $\gamma_J$ for each
value of $\gamma_d$ in the grid. Step 2 was performed using the method
described in Section~\ref{sec:nomregret}. The bisection was stopped
when the interval of lower/upper bounds on $\gamma_J$, denoted
$[\underline{\gamma}_J,\bar{\gamma}_J]$, satisfied
$\bar{\gamma}_J-\underline{\gamma}_J \le 10^{-4} + 10^{-4}
\bar{\gamma}_J$. The twenty points were computed in $\approx 6.4$sec
on a standard laptop, i.e. each point on the nominal regret curve took
$\approx 0.32$sec to compute.

\begin{figure}[h!]
  \centering
  \includegraphics[width=0.45\textwidth]{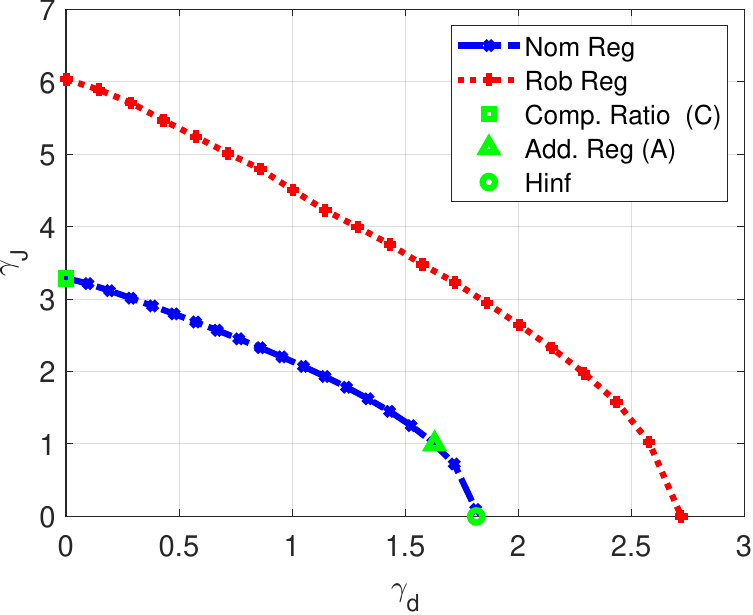}
  \caption{Pareto optimal curve for nominal and robust regret.}
  \label{fig:SimpleEx_Pareto}
\end{figure}

The left side of Figure~\ref{fig:SimpleEx_Nom} shows the frequency
responses for the nominal regret-optimal controllers along the Pareto
curve.  The $H_\infty$ and competitive ratio controllers are
highlighted in red-dashed and blue solid, respectively.  These two
controllers lie on the endpoints of the Pareto curve.  The other
Pareto controllers (black dotted) transition between these two
extremes for this example. It is notable that the competitive ratio
controller has a larger low frequency gain and smaller high frequency
gain than the other controllers.

\begin{figure}[h!]
  \centering
  \includegraphics[width=0.45\textwidth]{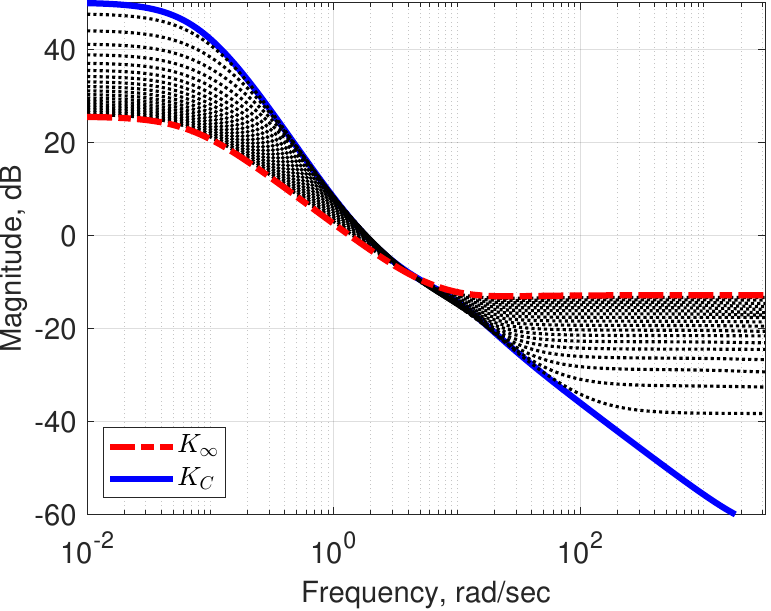}  
  \includegraphics[width=0.45\textwidth]{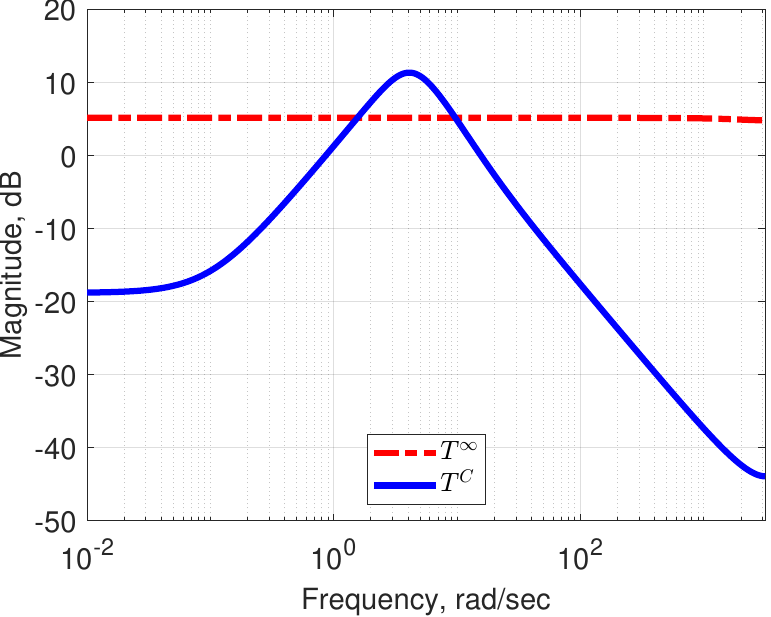}
  \caption{Left: Frequency response of nominal regret optimal
    controllers along the Pareto curve.  Right: Closed-loop
    responses of $H_\infty$ and competitive ratio controllers:
      $\| T^\infty(j\omega)\|_2$ and $\| T^C(j\omega)\|_2$ vs. $\omega$.}
    \label{fig:SimpleEx_Nom}
\end{figure}

To further interpret these results, let $T^\infty$ and $T^C$ denote
the closed-loop system from $d$ to $e$ obtained with the $H_\infty$
and competitive ratio controllers.  These are special cases of nominal
$(\gamma_d,\gamma_J)$-regret are discussed in
Section~\ref{sec:nomprob}.  The right side of
Figure~\ref{fig:SimpleEx_Nom} shows both $\|T^{\infty}(j\omega)\|_2$
and $\|T^{\infty}(j\omega)\|_2$ versus $\omega$. The closed-loop with
the $H_\infty$ controller has an approximately flat response.  This
flat response arises because the $H_\infty$ bound, expressed in the
frequency domain, corresponds to:
$\|T^{\infty}(j\omega)\|_2 \le \gamma_\infty$ $\forall \omega$ where
$\gamma_\infty=1.82=5.2$dB. On the other hand, the closed-loop
response with the competitive ratio controller has a similar shape to
that obtained with the non-causal controller in
Figure~\ref{fig:SimpleEx_NoncausalCost}.  This shape arises because
the competitive ratio bound corresponds to:
$\|T^C(j\omega)\|_2 \le \gamma_C \|T^o(j\omega)\|_2$ $\forall \omega$
where $\gamma_C = 3.29= 10.3$ dB.  The optimal non-causal controller
has good performance at low and high frequencies as noted above.
Hence the competitive ratio controller has similar characteristics
(although degraded by a factor of $\gamma_C=3.29$ relative to $K^o$).
This is the reason that $K^C$ has larger low frequency gain and rolls
off more rapidly at high frequencies than $K^\infty$ as shown on the
left side of Figure~\ref{fig:SimpleEx_Nom}.

Finally, we consider robust regret-optimal controllers.
Figure~\ref{fig:SimpleEx_Pareto} also shows the Pareto optimal curve
for feasible values $(\gamma_d,\gamma_J)$ of robust regret (red dotted
curve). The curve was generated via bisection on $\gamma_J$ using the
same grid of 20 values for $\gamma_d$. The bisection was performed
using the method described in Section~\ref{sec:robregret} with the
same stopping conditions.  The twenty points were computed in
$\approx 998$sec on a standard laptop, i.e. each point on the robust
regret curve took $\approx 50$sec to compute. The robust regret is
more computationally intensive as DK-synthesis must be used at each
bisection step to synthesize a controller that satisfies the robust
performance condition in Theorem~\ref{thm:MDregbnd}.

We first consider the competitive ratio controllers to study the
effect of the robust design.  The left side of
Figure~\ref{fig:SimpleEx_CRCosts} shows the nominal and robust
competitive ratio controllers.  The nominal competitive ratio
controller (blue solid) has large gain at low frequencies and a steep
transition to small gain at high frequencies.  The corresponding
nominal loop $G(s) A_0(s) K^C(s)$ has a phase margin of $36^\circ$ at
$3.7$\,rad/sec.  The robust competitive ratio controller (red dashed)
has a more shallow slope near 3.7 rad/sec. This corresponds to a less
negative phase (not shown) by the Bode Gain-Phase theorem. In fact,
the nominal loop with the robust competitive ratio controller has a
larger phase margin of $50^\circ$ at crossover frequency of 3.5
rad/sec.

\begin{figure}[h!]
  \centering
  \includegraphics[width=0.45\textwidth]{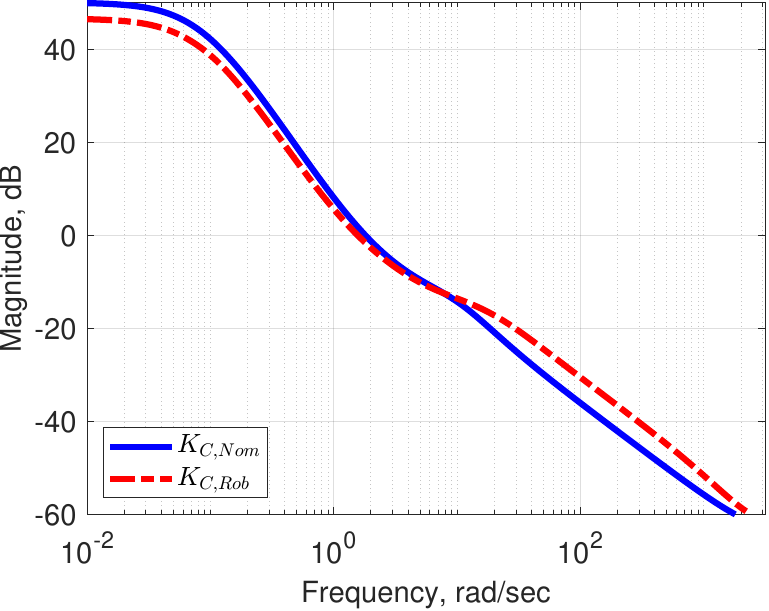}
  \includegraphics[width=0.45\textwidth]{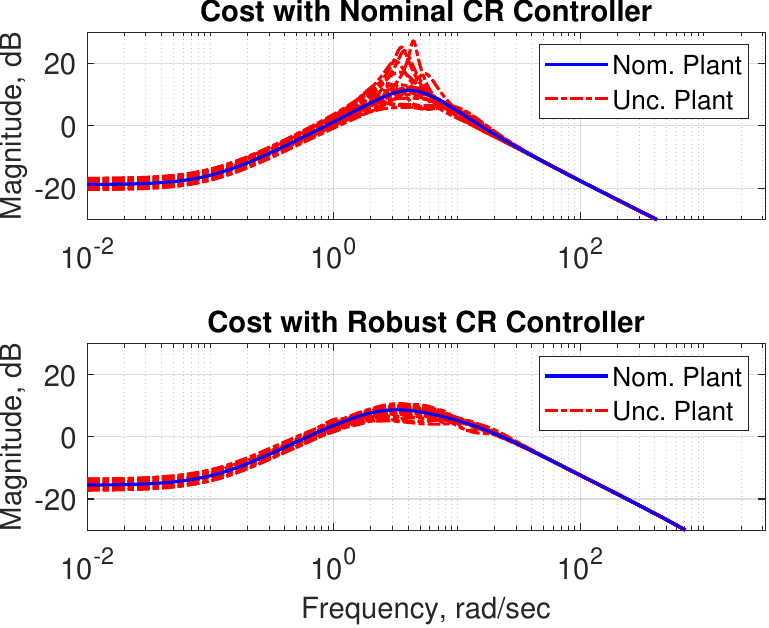}
  \caption{Left: Frequency response of nominal and robust
    competitive ratio controllers.  Right: Closed-loop
    responses with uncertain plant using the nominal (top)
    and robust (bottom)  competitive ratio controllers.}
\label{fig:SimpleEx_CRCosts}
\end{figure}

The right side of Figure~\ref{fig:SimpleEx_CRCosts} provides
additional insight. The top subplot shows the closed-loop response
with the nominal competitive ratio controller. The blue curve is
generated using the nominal plant and is the same as shown on the
right side of \ref{fig:SimpleEx_Nom}.  The red dashed curves are
generated on samples of the uncertain plant. Specifically, twenty
stable, fifth-order, LTI systems $\{\Delta_i\}_{i=1}^{20}$ are
randomly sampled and scaled to satisfy $\|\Delta_i\|_\infty \le
1$. These are substituted into the uncertain actuator model
\eqref{eq:SimpleA} and used to generate a sample of the uncertain
closed-loop.  The performance of the nominal competitive ratio
controller degrades significantly in the middle frequency range due to
this uncertainty.  The lower subplot is similar but is generated using
the robust competitive ratio controller. This controller is much less
sensitive to the model uncertainty as expected.

Other controllers along the nominal and robust Pareto front can be
compared. For example, the points along the horizontal axis of
Figure~\ref{fig:SimpleEx_Pareto} correspond to the nominal and robust
$H_\infty$ controllers.  The robust $H_\infty$ controller is a variant
of $\mu$-synthesis that minimizes the worst-case gain from $d$ to $e$
over all norm-bounded uncertainties $\| \Delta\|_\infty \le 1$. The
nominal and robust $H_\infty$ controllers are similar for this
particular example (although the robust cost is larger).  The reason
is that the nominal $H_\infty$ controller, shown on the left of
Figure~\ref{fig:SimpleEx_Nom} has relatively smaller gain at low
frequencies and relatively larger gain at high frequencies. Hence this
controller is already relatively robust. For example, the loop with
the nominal $H_\infty$ controller has a phase margin of $78^\circ$
with a crossover frequency of 3.6 rad/sec. The robust $H_\infty$
controller is similar with a slightly larger phase margin of
$85^\circ$ and a slightly lower crossover frequency of 3.2 rad/sec.
Both the nominal and robust $H_\infty$ controllers achieve robustness
but at the expense of sacrificing performance relative to the
nominal and competitive ratio controllers.

\subsection{Boeing 747}

A simplified (linearized) model for the longitudinal dynamics of a
Boeing 747 at one steady, level flight condition are given by (Problem
17.6 in \cite{boyd18}):
\begin{align}
  \label{eq:Boeing747Dyn}
  x_{t+1} = A x_t + B \tilde{u}_t + d_t,
\end{align}
where the state matrices are:
\begin{align*}
  A = \bsmtx 0.99 & 0.03 & -0.02 & -0.32 \\ 
               0.01 & 0.47 & 4.7 &  0 \\
               0.02 & -0.06 & 0.4 & 0 \\
               0.01 & -0.04 & 0.72 & 0.99 \esmtx 
\,\, \mbox{ and } \,\,
  B =  \bsmtx 0.01 & 0.99 \\ -3.44 & 1.66 \\ 
               -0.83 & 0.44 \\ -0.47 & 0.25 \esmtx.
\end{align*}
We will assume there is a multiplicative uncertainty at the plant input:
\begin{align}
\label{eq:Boeing747Unc}
  \tilde{u}= (I_2 + 0.6 \Delta) u,
\end{align}
where $u_t\in \R^2$ is the (nominal) command from the controller and
$\Delta$ is a $2\times 2$, stable, LTI uncertainty such that
$\| \Delta \|_\infty \le 1$.  Multiplicative uncertainty is a commonly
used to account for non-parametric (dynamic) modeling
errors~\cite{zhou96}.  No assumptions are made regarding the
state-dimension of $\Delta$.  The constant factor of 0.6 implies that
the effect of the uncertainty can be up to 60\% of the size of the
nominal control at each frequency. This factor could be made
frequency-dependent to account for increased non-parametric errors at
high frequencies. We'll use the constant factor for simplicity.

Define the generalized error as $e_t=\bsmtx x_t \\ u_t \esmtx$.  The
corresponding per-step cost is
$e_t^\top e_t = x_t^\top x_t + u_t^\top u_t$. We'll also assume the
controller has access to full information measurements, i.e.
$y_t = \bsmtx x_t \\ d_t \esmtx$. Based on this information, we can
formulate the robust regret optimal control design as in
Figure~\ref{fig:RobRegIC}. Specifically, define $(v,w)$ as the
input/output signals of the uncertainty in \eqref{eq:Boeing747Unc},
i.e. $w=\Delta v$ with $v=u$.  Then substitute $\tilde{u}= u + 0.6 w$
into \eqref{eq:Boeing747Dyn}:
\begin{align}
  x_{t+1} = A x_t + (0.6B) w_t + d_t + B u_t.
\end{align}
This has the form of the state update equation for
the uncertain plant $P$ in \eqref{eq:Punc} with $B_w=0.6B$, 
$B_d=I_4$, and $B_u = B$.  The output equations for $(v,e,y)$ in
\eqref{eq:Punc} can be derived similarly.  Note that the nominal
dynamics are given by $\Delta=0$ and neglecting the $(v,w)$ channels.

This example was used in prior work on competitive ratio
\cite{goel22TAC} and additive regret \cite{sabag21ACC,sabag21arXiv}
with nominal dynamics ($\Delta=0$).  The nominal regret, as defined in
Section~\ref{sec:nomregret}, includes three special optimal
controllers (values within numerical tolerances): (a) $H_\infty$:
$(\gamma_d,\gamma_J) = (28.47,0)$, (b) competitive ratio:
$(\gamma_d,\gamma_J) = (0,1.33)$, and (c) additive regret:
$(\gamma_d,\gamma_J) = (12.27,1)$. The competitive ratio reported
in \cite{goel22TAC} is $\gamma_C^2=1.77 (=1.33^2)$ and this agrees
with the value computed here.

Figure~\ref{fig7:Boeing747RobReg} shows these three points along with
the Pareto optimal curve for feasible values $(\gamma_d,\gamma_J)$ of
nominal regret (blue dashed curve). The curve was generated by: (i)
selecting 20 values of $\gamma_d$ in the range of
$[0.001,\, 0.999]\gamma_\infty$ where $\gamma_\infty = 28.47$, and
(ii) bisecting to find the minimal value of $\gamma_J$ for each value
of $\gamma_d$ in the grid. Step 2 was performed using the method
described in Section~\ref{sec:nomregret}. The bisection was stopped
when the interval of lower/upper bounds on $\gamma_J$, denoted
$[\underline{\gamma}_J,\bar{\gamma}_J]$, satisfied
$\bar{\gamma}_J-\underline{\gamma}_J \le 10^{-2} + 10^{-3}
\bar{\gamma}_J$. The twenty points were computed in $\approx 4.5$sec
on a standard laptop, i.e. each point on the nominal regret curve took
$\approx 0.45$sec to compute.

Figure~\ref{fig7:Boeing747RobReg} also shows the Pareto optimal curve
for feasible values $(\gamma_d,\gamma_J)$ of robust regret (red dotted
curve). The curve was generated via bisection on $\gamma_J$ using the
same grid of 20 values for $\gamma_d$. The bisection was performed
using the method described in Section~\ref{sec:robregret} with the
same stopping conditions.  The twenty points were computed with
$\approx 789$sec on a standard laptop, i.e. each point on the robust
regret curve took $\approx 39.4$sec to compute. The robust regret is
more computationally intensive as DK-synthesis must be used at each
bisection step to synthesize a controller that satisfies the robust
performance condition in Theorem~\ref{thm:MDregbnd}.

The several key points from this example. First, our notion of nominal
regret (Section~\ref{sec:nomregret}) contains $H_\infty$, competitive
ratio, and additive regret as special cases. Second, robust regret
(Section~\ref{sec:robregret}) can be used to synthesis robust
controllers and assess the impact of model uncertainty. For example,
the red dotted and blue dashed curves in
Figure~\ref{fig7:Boeing747RobReg} are similar for larger values of
$\gamma_d$ but diverge for smaller values of $\gamma_d$. This
indicates that the competitive ratio controller is more sensitive to
the non-parametric errors than the $H_\infty$ controller. These results
are specific to this particular problem (plant dynamics, cost, and
full-information measurement). Hence no general conclusions should
be drawn regarding the robustness properties of the various controllers.

% Instead, the main takeaway is that robust regret provides a means to
% account for model uncertainty.

% 10 points = [tNom tRob] ~= [2.9169  457.2330]
% 20 points = [tNom,tRob] ~= [6.8657 788.8886]  ( tRob/20 ~=   39.44)
% The nominal takes ~4.5sec with N=20 if I run the m-file several times 
% with the laptop having been at rest (no background processes). The
% average time is ~0.2sec.

\begin{figure}[h!]
  \centering
  \includegraphics[width=0.45\textwidth]{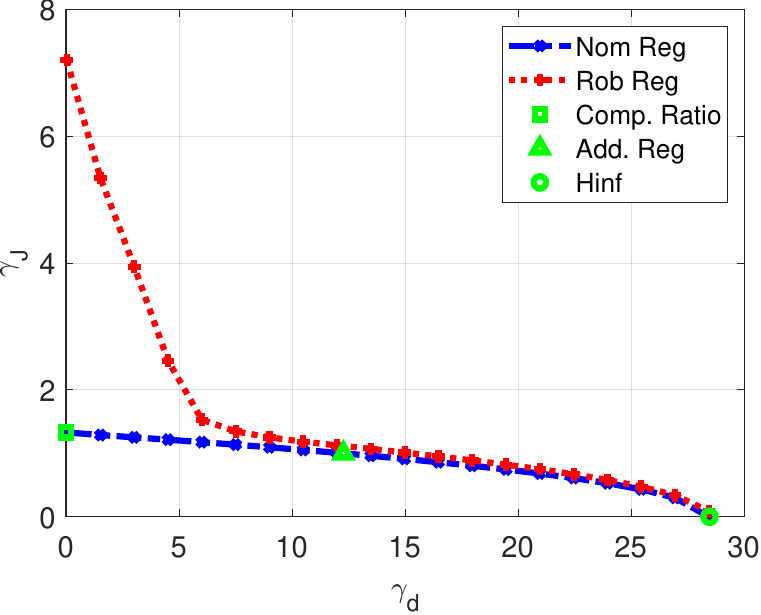}
  \caption{Pareto optimal curve for nominal and robust regret for
  Boeing 747.}
\label{fig7:Boeing747RobReg}
\end{figure}

\subsection{Quarter Car Example}

In this section we consider the design of an active suspension
controller for a quarter car model. The design is based on the
Matlab demo~\cite{matlab22} and a summary of the objectives can
be found in the corresponding tutorial video~\cite{douglas20}.
A model for the quarter car is given by:
\begin{align}
  \dot{x}(t) & = A \, x(t) + B \, \bsmtx r(t) \\ f_s(t) \esmtx \\
     \bsmtx a_b(t) \\ s_d(t) \esmtx  & = C \, x(t) + D \, 
                \bsmtx r(t) \\ f_s(t) \esmtx,
\end{align}
where $r(t)\in \R$ is the road disturbance (m), $f_s(t)\in \R$ is the
active suspension force (kN), $a_b(t) \in \R$ is the car body
acceleration (m/s$^2$), $s_d(t) \in \R$ is the suspension displacement
(m), and $x(t) \in \R^4$ is the state.  The state-space data is:
\begin{align*}
& A  =  \bsmtx 0 & 1 & 0 & 0 \\ 
  -\frac{k_s}{m_b} &  -\frac{b_s}{m_b} & \frac{k_s}{m_b} & \frac{b_s}{m_b} \\
   0 & 0 & 0 & 1 \\
  \frac{k_s}{m_w} &  \frac{b_s}{m_w} & -\frac{k_s+k_t}{m_w} & -\frac{b_s}{m_w} 
\esmtx,
&& B = \bsmtx 0 & 0 \\ 0 & \frac{10^3}{m_b} \\ 0 & 0 \\
        \frac{k_t}{m_w} & -\frac{10^3}{m_w} \esmtx, \\
& C = \bsmtx   
   -\frac{k_s}{m_b} &  -\frac{b_s}{m_b} & \frac{k_s}{m_b} & \frac{b_s}{m_b} \\
    1 & 0 & -1 & 0 \esmtx, 
&& D = \bsmtx 0 & \frac{10^3}{m_b} \\ 0 & 0 \esmtx,
\end{align*}
where $m_b = 300$kg, $m_w = 60$kg, $b_s = 1000$N/m/s,
$k_s = 16000$N/m, and $k_t = 1.9\times 10^5$N/m. The model has lightly
complex poles at $-1.43\pm 6.91j$ and $-8.57\pm 57.6j$. The goal is to
design the active suspension controller to balance driver comfort
(small $a_b$) and road handling (small $s_d$).

A hydraulic actuator generates the force based on a control input $u$.
The nominal hydraulic actuator model is first-order with a time
constant of $\frac{1}{60}$ sec, i.e. $A_0(s) = \frac{60}{s+60}$.  The
maximum actuator displacement is 0.05m. This places a constraint on
the suspension displacement to be less than 0.05m. The actuator
dynamics are assumed to uncertain. This is modeled with an input
multiplicative uncertainty:
\begin{align*}
  A(s) = A_0(s) \cdot (1 + W_{unc}(s) \Delta(s) ),
\end{align*}
where $W_{unc}(s) = \frac{3s+18.5}{s+46.3}$ and $\Delta(s)$ is a
stable, LTI system satisfying $\|\Delta\|_\infty \le 1$.  The weight
$W_{unc}(s)$ has DC gain of 0.4, $|W_{unc}(j\infty)|=3$, and a zero
near $-6.2$rad/sec.  This represents a relatively small uncertainty
(40\%) uncertainty at low frequencies with increasing uncertainty
above $6.2$rad/sec. The continuous-time quarter car and hydraulic
actuator models are discretized using zero-order hold and sample time
$T_s = 0.002$ sec.\footnote{$A_0(s)$ and $W_{unc}(s)$ are discretized
  with sample time $T_s$. The normalized uncertainty is replaced by a
  discrete-time, stable LTI system $\Delta(z)$ with sample time $T_s$
  satisfying $\|\Delta\|_\infty \le 1$.}

Figure~\ref{fig:qcarIC} shows the interconnection for the control
design. There are three disturbances associated with the road input
and measurement noise.  Each of these has a constant weights that
(roughly) model the amplitudes of the corresponding input
disturbances: $W_{road} = 0.07$, $W_{d_2} = 0.01$, and
$W_{d_3} = 0.5$. There are three generalized errors that model the
competing objectives of actuator commands, body acceleration, and
suspension travel. The corresponding performance weights are:
\begin{align*}
 W_{act}(s) & =  0.8\, \frac{s+50}{s+500}, \\
 W_{s_d}(s) & = \frac{0.00625 s + 0.5}{0.005 s + 0.04}, \\
 W_{a_b}(s) & = \frac{0.03s + 4.5}{8s + 3.6}.
\end{align*}
These weights are also discretized using a zero-order hold with
$T_s=0.002$sec. These weights correspond to a ``balanced'' objective
between comfort and handling.  Additional details on the design
objectives can be found in \cite{matlab22,douglas20}.

\begin{figure}[h]
\centering
\scalebox{0.9}{
\begin{picture}(260,120)(0,-62)
 \thicklines
 % I/O for Plant
 \put(100,0){\framebox(50,50){1/4 Car}}
 \put(75,10){\vector(1,0){25}}  
 \put(83,15){$f_s$}
 \put(75,40){\vector(1,0){25}}  
 \put(83,45){$r$}
 \put(150,10){\vector(1,0){50}}  
 \put(157,15){$s_d$}
 \put(150,40){\vector(1,0){50}}  
 \put(177,45){$a_b$}
 % Weights on input signals
 \put(43,-2){\framebox(32,24){Act($\Delta$)}}
 \put(13,10){\vector(1,0){30}}  
 \put(4,8){$u$}
 \put(43,28){\framebox(32,24){$W_{road}$}}
 \put(13,40){\vector(1,0){30}}  
 \put(3,38){$d_1$}
 % Weights on output signals
 \put(200,-2){\framebox(32,24){$W_{s_d}$}}
 \put(232,10){\vector(1,0){20}}  
 \put(255,8){$e_3$}
 \put(200,28){\framebox(32,24){$W_{a_b}$}}
 \put(232,40){\vector(1,0){20}}  
 \put(255,38){$e_2$}
 % Weight on actuator
 \put(25,10){\line(0,-1){30}}  
 \put(25,-20){\vector(1,0){18}}  
 \put(43,-32){\framebox(32,24){$W_{act}$}}
 \put(75,-20){\vector(1,0){20}}  
 \put(97,-22){$e_1$}
 % First meas with noise
 \put(160,10){\vector(0,-1){27}}
 \put(160,-20){\circle{6}}
 \put(157,-20){\vector(-1,0){17}}  
 \put(130,-22){$y_1$}
 \put(200,-20){\vector(-1,0){37}}  
 \put(200,-32){\framebox(32,24){$W_{d_2}$}}
 \put(252,-20){\vector(-1,0){20}}  
 \put(255,-22){$d_2$}
 % Second meas with noise
 \put(180,40){\line(0,-1){28}}
 \put(180,8){\line(0,-1){26}}
 \put(180,-22){\vector(0,-1){25}}
 \put(180,-50){\circle{6}}
 \put(177,-50){\vector(-1,0){37}}  
 \put(130,-52){$y_2$}
 \put(200,-50){\vector(-1,0){17}}  
 \put(200,-62){\framebox(32,24){$W_{d_3}$}}
 \put(252,-50){\vector(-1,0){20}}  
 \put(255,-52){$d_3$}
\end{picture}
} % End scalebox
\caption{Feedback interconnection for quarter-car disturbance rejection.}
\label{fig:qcarIC}
\end{figure}
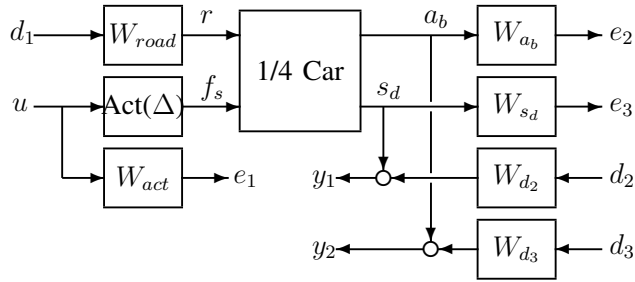

The active suspension design (Figure~\ref{fig:qcarIC}) falls within
the general nominal design interconnection framework
(Figure~\ref{fig:FLPKUnc}).  We first designed a nominal additive
regret controller.  This design was performed on the interconnection
with no uncertainty so that the actuator is at the nominal dynamics,
$A(s)=A_0(s)$. The optimal additive regret controller achieved (within
a bisection tolerance) a regret of $(\gamma_d,\gamma_J)=(0.43,1)$.
\footnote{The optimal nominal $H_\infty$ controller achieved
  $\gamma_\infty =0.66$.} We also designed an optimal robust, additive
regret controller including the actuator uncertainty.  The robust
additive regret controller achieved $(\gamma_d,\gamma_J)=(0.78,1)$.
Figure~\ref{fig:ActiveSuspFD} shows Bode magnitude plots from road
input to body acceleration (without the various weights shown in
Figure~\ref{fig:qcarIC}). The plot shows responses for the
open-loop (OL) and closed-loop (CL) with nominal additive regret (AR)
and robust additive regret (Rob) controllers. The lightly damped
suspension modes are evident in the open-loop response near $7$ and
$58$ rad/sec.  Both the nominal and robust additive regret controllers
reduce the effect of the first open loop mode near 7 rad/sec.  The
other mode near 58 rad/sec cannot be reduced due to a transmission
zero in the plant model.

\begin{figure}[h!]
  \centering
  \includegraphics[width=0.4\textwidth]{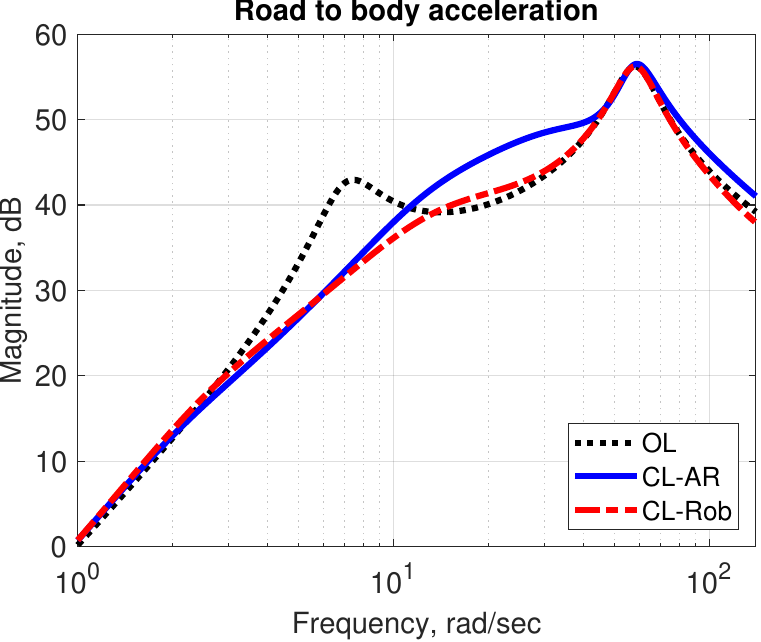}
  \caption{Bode magnitude plots from road input to body acceleration
    output for the open-loop (OL) and closed-loop (CL) with nominal
    additive regret (AR), and robust additive regret (Rob)
    controllers.}
  \label{fig:ActiveSuspFD}
\end{figure}

Figure~\ref{fig:ActiveSuspNomTD} shows a closed-loop, time-domain
response with the \emph{nominal} additive regret controller.  The road
disturbance is $r(t) = 0.025 \cdot (1-\cos(8\pi t))$ for
$0\le t \le 0.2$ and $r(t)=0$ for $t>0.2$.  The simulations are
performed with the nominal actuator model and 50 samples of the
uncertain actuator model. The nominal responses are good for both the
body acceleration and suspension displacement.  The displacement
remains within the allowed $\pm 0.05$m of travel.  However, the
simulations with the uncertain actuator show a wide variability
including lightly damped, oscillatory responses.

\begin{figure}[h!]
  \centering
  \includegraphics[width=0.4\textwidth]{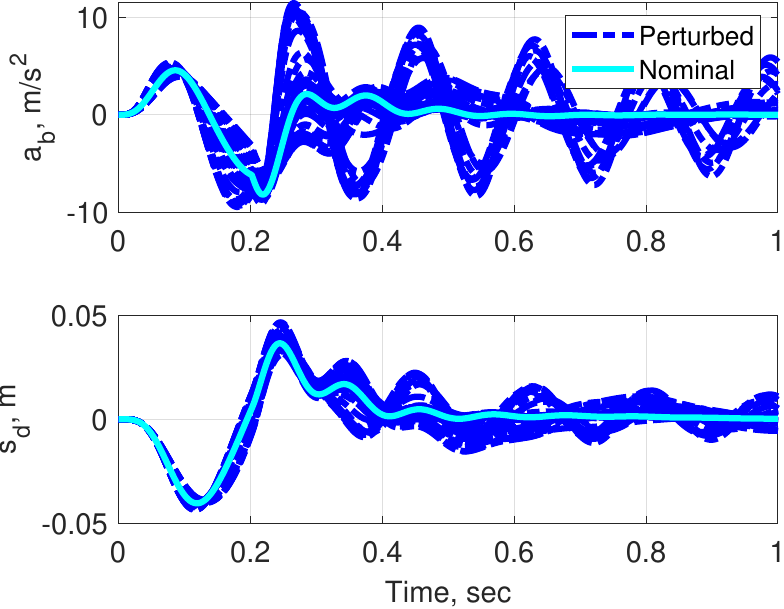}
  \caption{Time domain plots of  body acceleration and suspension travel
    with a road disturbance input  and the optimal additive regret controller.
    Simulations shown with the nominal plant dynamics and 50 samples
    of the uncertain actuator model.}
  \label{fig:ActiveSuspNomTD}
\end{figure}

Figure~\ref{fig:ActiveSuspRobTD} shows the closed-loop, time-domain
responses with the \emph{robust} additive regret controller. The
responses with the nominal actuator model are similar and even
slightly better with the body acceleration.  The main benefit of the
robust additive regret controller is that the simulations with the
uncertain actuator model show much less variability. Thus this
controller is more robust to the actuator uncertainty as expected.

\begin{figure}[h!]
  \centering
  \includegraphics[width=0.4\textwidth]{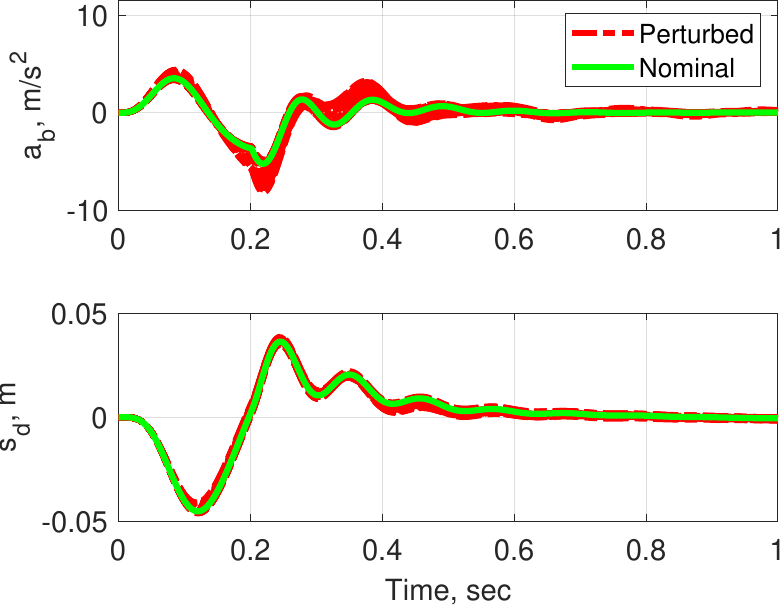}
  \caption{Time domain plots of body acceleration and suspension
    travel with a road disturbance input and the optimal \emph{robust}
    additive regret controller.  Simulations shown with the nominal
    plant dynamics and 50 samples of the uncertain actuator model.}
  \label{fig:ActiveSuspRobTD}
\end{figure}

%---------------------------------------------------------
\section{Conclusions}

This paper presents a synthesis method for robust regret optimal
control for uncertain, discrete-time, LTI systems. The baseline for
performance is the optimal non-causal controller designed on the
nominal plant model and with knowledge of the disturbance.  Robust
regret is defined as the performance of a given controller relative to
this optimal non-causal control. Our definition of regret includes
previous definitions of (additive) regret, (multiplicative)
competitive ratio, and $H_\infty$ performance.  We show that a
controller achieves robust regret if and only if it satisfies a robust
$H_\infty$ performance condition. DK-iteration can be used to
synthesize a controller that satisfies this condition and hence
achieve a given level of robust regret.  The approach is demonstrated
via three examples.  Future work will consider the design of
robust regret controllers where the baseline depends on the specific
value of the model uncertainty, e.g. possibly using results from
linear parameter varying design. Future work will also include
comparisons of this approach with optimal regret controllers designed
via online convex optimization and disturbance-action controllers.

%---------------------------------------------------------
% \section*{Acknowledgments}
% This work was partially supported by the National Science Foundation
% under Grant No. NSF-CMMI-1254129 entitled ``CAREER: Probabilistic
% Tools for High Reliability Monitoring and Control of Wind Farms.''  

%------------------------------------------------------
\bibliographystyle{IEEEtran}
\bibliography{RobustRegretControl}

%------------------------------------------------------
\appendix

\subsection{Discrete-Time Spectral Factorization}
\label{app:DTSF}

This appendix summarizes existing results for discrete-time spectral
factorizations.  Good overviews can be found in Section 8.3 of
\cite{kailath00} or Chapter 13 of \cite{hassibi99}. The specific
construction given here for the spectral factorization, stated as
Lemma~\ref{lem:DTSF} below, is based on results in
\cite{katayama92}. 

% This is used to factorize the regret boundt in
% Section~\ref{sec:nomfeas}.

Consider a discrete-time, linear time-invariant system $\hat{P}$ with
the following state-space representation:
\begin{align}
  \label{eq:Phat}
  \begin{split}
    \hat{x}_{t+1} & = \hat{A} \, \hat{x}_t + \hat{B} \, d_t \\
    \hat{e}_t & = \hat{C} \, \hat{x}_t + \hat{D} \, d_t.
  \end{split}
\end{align}
The input/output dimensions of $\hat{P}$ are $n_{\hat{e}} \times n_d$.
We assume that $\hat{P}$ is stable although not necessarily causal.
An input $d\in \ell_2$ generates a state $\hat{x} \in \ell_2$ and
output $\hat{e} \in \ell_2$. The boundary conditions for the state are
$\hat{x}_{-\infty}=0$ and $\hat{x}_\infty=0$.\footnote{The state for
  the causal (stable) dynamics is initialized to 0 at $t=-\infty$ and
  converges back to 0 as $t\to \infty$ when $d\in \ell_2$.  Similarly,
  the non-causal (stable) dynamics are initialized to 0 at $t=+\infty$
  and converge back to 0 as $t\to -\infty$ when $d\in \ell_2$.}  We'll
denote $\hat{e}=\hat{P}d$ as the output generated by the input
$d\in \ell_2$ subject to these boundary conditions.  Let
$\hat{J}:= \langle \hat{e}, \hat{e} \rangle$ be the resulting cost.

The appendix is rewrites $\hat{J}$ in a more useful form
using two ingredients.  The first ingredient is the para-Hermitian
conjugate of $\hat{P}$, denoted $\hat{P}^\sim$, and defined as
follows:
\begin{align}
  \label{eq:PhatTil}
  \begin{split}
    \hat{A}^\top \bar{x}_{t+1} & = \bar{x}_t - \hat{C}^\top \,\bar{e}_t \\
    \bar{d}_t & = \hat{B}^\top \, \bar{x}_{t+1} + \hat{D}^\top \, \bar{e}_t.
  \end{split}
\end{align}
The input/output dimensions for $\hat{P}^\sim$ are
$n_d \times n_{\hat{e}}$.  If $\bar{e}_t \in \ell_2$ then the state of
$\hat{P}^\sim$ satisfies the boundary conditions $\bar{x}_{-\infty}=0$
and $\bar{x}_\infty=0$.  The para-Hermitian conjugate $\hat{P}^\sim$
also satisfies the adjoint property stated in the next lemma.

\begin{lemma}
\label{lem:PAdj}
The para-Hermitian conjugate $\hat{P}^\sim$ satisfies
satisfies $\langle \hat{P} d, \bar{e} \rangle
= \langle d, \hat{P}^\sim \bar{e} \rangle$ for all 
$d$, $\bar{e} \in \ell_2$.
\end{lemma}
\begin{proof}
First, substitute for the output equation of $\hat{P}$:
\begin{align*}
\langle \hat{P} d, \bar{e} \rangle =  
\sum_{t=-\infty}^\infty \left(\hat{C} \hat{x}_t + \hat{D} d_t \right)^\top 
\bar{e}_t.
\end{align*}
Next, substitute for $\hat{C}^\top \bar{e}_t$ using the
dynamics of $\hat{P}^\sim$:
\begin{align*}
\langle \hat{P} d, \bar{e} \rangle =  
\sum_{t=-\infty}^\infty \hat{x}_t^\top (\bar{x}_t-\hat{A}^\top \bar{x}_{t+1})  
+  d_t^\top \hat{D}^\top  \bar{e}_t.
\end{align*}
Finally, substitute for $\hat{A} \hat{x}_t$ using the dynamics
of $\hat{P}$:
\begin{align*}
\langle \hat{P} d, \bar{e} \rangle =  
\sum_{t=-\infty}^\infty \hat{x}_t^\top \bar{x}_t
- \hat{x}_{t+1}^\top \bar{x}_{t+1}  
+ d_t^\top (\hat{B}^\top \bar{x}_{t+1} + \hat{D}^\top \bar{e}_t).
\end{align*}
The first two terms are a telescoping sum. These terms sum to
zero by the boundary conditions $\hat{x}_{-\infty}=\hat{x}_\infty=0$.  
Substitute for $\bar{d}_t$ using the output equation for $\hat{P}^\sim$:
\begin{align*}
\langle \hat{P} d, \bar{e} \rangle 
= \sum_{t=-\infty}^\infty d_t^\top \bar{d}_t
= \langle d, \hat{P}^\sim \bar{e} \rangle.
\end{align*}
\end{proof}

The second ingredient used to rewrite the cost $\hat{J}$ is the
spectral factorization as defined below.

\begin{defin}
  %Consider the $n_{\hat{e}}\times n_d$ system $\hat{P}$ in \eqref{eq:Phat}.
  A square $n_d \times n_d$ system $F$ is a spectral factor of
  $\hat{P}^\sim \hat{P}$ if: (i) $\hat{P}^\sim \hat{P} = F^\sim F$,
  (ii) $F$ is stable, causal, and invertible, and (iii)
  $F^{-1}$ is  stable and causal.  Moreover, $F^\sim F$ is
  called a spectral factorization of $\hat{P}^\sim \hat{P}$.
\end{defin}

We can now use the two ingredients to rewrite the cost
$\hat{J}:= \langle \hat{e}, \hat{e} \rangle$.  First, use the adjoint
property in Lemma~\ref{lem:PAdj} to write the cost as:
$\hat{J} = \langle d, \hat{P}^\sim \hat{P} d \rangle$. Next, if a
spectral factor $F$ exists then $\hat{P}^\sim\hat{P} = F^\sim F$.
Thus another application of the adjoint property in
Lemma~\ref{lem:PAdj} yields $\hat{J}:= \langle F d, Fd \rangle$. In
other words, we can compute $\hat{J}$ using the spectral factor
$F$. The dynamics of $F$ are stable/causal and with a stable/causal
inverse.  The next lemma provides conditions for the existence of a
spectral factorization of $\hat{P}^\sim \hat{P}$. An explicit,
state-space construction is also given for the spectral factor (when
it exists). This is essentially a restatement of Theorems 4.1 and 4.2
in \cite{katayama92}.

\begin{lemma}
\label{lem:DTSF}
Let $(\hat{A},\hat{B},\hat{C},\hat{D})$ be given and define 
$\hat{Q}:=\hat{C}^\top \hat{C}$, $\hat{S}:= \hat{C}^\top \hat{D}$,
and $\hat{R} = \hat{D}^\top \hat{D}$.

Assume: (i) $\hat{R} \succ 0$, (ii) $(\hat{A},\hat{B})$ is
stabilizable, (iii) $\hat{A}-\hat{B}\hat{R}^{-1}\hat{S}^\top$ is
nonsingular, (iv)
$\bsmtx \hat{A}-e^{j\theta} I & \hat{B} \\ \hat{C} & \hat{D} \esmtx$
has full column rank for all $\theta \in [0,2\pi]$, (v) $\hat{A}$ has
no eigenvalues on the unit circle, and (vi) $\hat{A}$ is nonsingular.
Then:
\begin{enumerate}
\item There is a unique stabilizing solution $\hat{X} \ge 0$ to the
  following DARE:
  \begin{align}
    \label{eq:DAREXhat}
    0 = & \hat{X}  - \hat{A}^\top \hat{X} \hat{A}  - \hat{Q} \\
    \nonumber
        % & \, + (\hat{A}^\top \hat{X}\hat{B} + \hat{S})\, 
        %    (\hat{R} +\hat{B}^\top \hat{X}\hat{B})^{-1} \, 
        %   (\hat{A}^\top \hat{X}\hat{B} + \hat{S})^\top,
        & \, + (\hat{A}^\top \hat{X}\hat{B} + \hat{S})\, 
           \hat{H}^{-1} \, 
          (\hat{A}^\top \hat{X}\hat{B} + \hat{S})^\top,
  \end{align}
  where $\hat{H}:=\hat{R}+\hat{B}^\top \hat{X} \hat{B}\succ 0$. The gain
  $\hat{K}_x:=\hat{H}^{-1}(\hat{A}^\top \hat{X} \hat{B}+\hat{S})^\top$
  is stabilizing, i.e.  $\hat{A}-\hat{B}\hat{K}_x$ is a Schur matrix.
  Moreover, $(\hat{A}-\hat{B}\hat{K}_x)$ is nonsingular.

\vspace{0.05in}
\item There is a unique stabilizing solution $\hat{Y} \ge 0$ to the
  following DARE:
  \begin{align}
    \label{eq:DAREYhat}
    0  = \hat{Y}  - \hat{A} \hat{Y} \hat{A}^\top   
       + (\hat{A} \hat{Y}\hat{K}_x^\top) \, 
            \hat{W}^{-1} \,
          (\hat{A} \hat{Y}\hat{K}_x^\top)^\top,
  \end{align}  
  where
  $\hat{W}:=\hat{H}^{-1} + \hat{K}_x \hat{Y} \hat{K}_x^\top \succ 0$.
  The gain
  $\hat{K}_y:=\hat{W}^{-1}(\hat{A} \hat{Y} \hat{K}_x^\top)^\top$ is
  stabilizing, i.e.  $\hat{A}^\top-\hat{K}_x^\top\hat{K}_y$ is a Schur
  matrix.

\vspace{0.05in}
\item Define $A_F := \hat{A}-\hat{K}_y^\top\hat{K}_x$,
  $B_F := \hat{B}-\hat{K}_y^\top$,
  $C_F = \hat{W}^{-\frac{1}{2}} \hat{K}_x$, and
  $D_F:=\hat{W}^{-\frac{1}{2}}$.  A spectral factorization of
  $\hat{P}^\sim \hat{P}$ is $F^\sim F$ where the spectral factor $F$
  is:
  \begin{align}
    \label{eq:FSF}
    \begin{split}
    \tilde{x}_{t+1} & = A_F \, \tilde{x}_t  + B_F \, d_t \\
    \tilde{e}_{t} & = C_F \, \tilde{x}_t   +  D_F  \, d_t.
  \end{split}
  \end{align} 
\end{enumerate}
\end{lemma}
\begin{proof}
  First, note that Statement 1 corresponds to the DARE in
  Lemma~\ref{lem:DARE} with $(A,B_u,C_e,D_{eu},Q$, $R,S)$ replaced
  by $(\hat{A},\hat{B},\hat{C},\hat{D},\hat{Q},\hat{R},\hat{S})$.
  Thus Statement 1 follows from assumptions (i)-(iv) and
  Lemma~\ref{lem:DARE}.

  Next, note that Statement 2 corresponds to the DARE in
  Lemma~\ref{lem:DARE} with $(A,B_u,C_e,D_{eu},Q,R,S)$ replaced by
  $(\hat{A}^\top,
  \hat{K}_x^\top,0,\hat{H}^{-\frac{1}{2}},0,\hat{H}^{-1},0)$.
  Moreover, we have: (a) $\hat{H}^{-1} \succ 0$, (b)
  $(\hat{A}^\top,\hat{K}_x^\top)$ is stabilizable because
  $(\hat{A}-\hat{B}\hat{K}_x)^\top$ is a Schur matrix, (c) $\hat{A}$
  is nonsingular by assumption (vi), and (d)
  $\bsmtx \hat{A}^\top-e^{j\theta} I & \hat{K}_x^\top \\ 0 &
  \hat{H}^{-\frac{1}{2}} \esmtx$ has full column rank for all
  $\theta \in [0,2\pi]$ by assumption (v) and
  $\hat{H}^{-\frac{1}{2}}\succ 0$.  Thus Statement 2 follows from
  assumptions (i)-(vi) and Lemma~\ref{lem:DARE}.

  For statement 3, we first show that
  $\hat{P}^\sim \hat{P} = F^\sim F$. Use the output equation
  of $\hat{P}$ to express the cost as:
  \begin{align*}
    \hat{J} = \langle \hat{e}, \hat{e} \rangle
       = \sum_{t=-\infty}^\infty \bmtx \hat{x}_t \\ d_t \emtx^\top 
            \bmtx \hat{Q} &  \hat{S} \\  \hat{S}^\top &  \hat{R} \emtx 
             \bmtx \hat{x}_t \\ d_t \emtx.
  \end{align*}
  Substitute for $\hat{Q}$ using the DARE \eqref{eq:DAREXhat} to show:
  \begin{align*}
    \bsmtx \hat{Q} &  \hat{S} \\  \hat{S}^\top &  \hat{R} \esmtx =
    \bsmtx \hat{K}_x^\top \\ I \esmtx \hat{H} \bsmtx \hat{K}_x & I \esmtx
    + \bsmtx \hat{X} & 0 \\ 0 & 0 \esmtx   
    -\bsmtx \hat{A}^\top \\ \hat{B}^\top \esmtx 
         \hat{X} \bsmtx \hat{A} & \hat{B} \esmtx.
  \end{align*}
  Thus the cost is equal to:
{\footnotesize
  \begin{align*}
    \hat{J} = \sum_{t=-\infty}^\infty \hat{x}_t^\top \hat{X} \hat{x}_t
             - \hat{x}_{t+1}^\top \hat{X} \hat{x}_{t+1} +      
    \bmtx \hat{x}_t \\ d_t \emtx^\top 
    \bsmtx \hat{K}_x^\top \\ I \esmtx \hat{H} \bsmtx \hat{K}_x & I \esmtx
    \bmtx \hat{x}_t \\ d_t \emtx.
  \end{align*}
}

\noindent
The first two terms are a telescoping sum. These terms sum to zero by
the boundary conditions $\hat{x}_{-\infty}=\hat{x}_\infty=0$.  Thus
the cost is equal to $\hat{J} = \langle \hat{F} d,\hat{F} d \rangle$
where $\hat{F}$ is the following intermediate system:
  \begin{align}
    \label{eq:Fhat}
    \begin{split}
      \hat{x}_{t+1} & = \hat{A} \, \hat{x}_t + \hat{B} \, d_t \\
      \tilde{e}_t &  =  \hat{H}^{\frac{1}{2}} \hat{K}_x \, \hat{x}_t 
             + \hat{H}^{\frac{1}{2}} \, d_t.
    \end{split}
  \end{align}
  Note that $\hat{F}$ is a square $n_d\times n_d$ system and
  invertible. The state matrix for $\hat{F}^{-1}$ is
  $\hat{A}-\hat{B}\hat{K}_x$ which is a Schur matrix by Statement 1.
  Hence $\hat{F}^{-1}$ is stable and causal.  However, $\hat{A}$ may
  not necessarily be a Schur matrix so one additional step is required
  to obtain the spectral factor $F$.

  By Lemma~\ref{lem:PAdj}, the cost is equal to
  $\hat{J} = \langle d, \hat{F}^\sim \hat{F} d \rangle$.
  The matrix $\hat{A}$ is assumed to be nonsingular.  This can be used
  to express the para-Hermitian conjugate $\hat{F}^\sim$ as:
  \begin{align}
    \label{eq:FhatTil}
    \begin{split}
      \bar{x}_{t+1} & =  \hat{A}^{-\top} \, \bar{x}_t -
      \hat{A}^{-\top} \hat{K}_x^\top \hat{H}^{\frac{1}{2}} \,\bar{e}_t \\
      \bar{d}_t & = (\hat{B}^\top \hat{A}^{-\top}) \, \bar{x}_t 
           + (I-\hat{B}^\top \hat{A}^{-\top} \hat{K}_x^\top)
             \hat{H}^{\frac{1}{2}} \, \bar{e}_t.
           \end{split}
  \end{align}
  The system $\hat{F}^\sim \hat{F}$ corresponds to the serial
  connection of $\hat{F}$ and $\hat{F}^\sim$, i.e. combine
  \eqref{eq:Fhat} and \eqref{eq:FhatTil} with
  $\bar{e}_t = \tilde{e}_t$.  This yields a state-space model for
  $\hat{F}^\sim \hat{F}$:
  {\footnotesize
  \begin{align*}
    \bmtx \hat{x}_{t+1} \\ \bar{x}_{t+1} \emtx & = 
     \bmtx \hat{A} & 0  \\ -\hat{A}^{-\top}\hat{K}_x^\top\hat{H}\hat{K}_x 
    & \hat{A}^{-\top} \emtx \, \bmtx \hat{x}_t \\ \bar{x}_t \emtx 
   + \bmtx \hat{B} \\ -\hat{A}^{-\top}\hat{K}_x^\top\hat{H} \emtx
     d_t \\
   \bar{e}_t & 
    = \bmtx (I-\hat{B}^\top \hat{A}^{-\top} \hat{K}_x^\top)\hat{H}\hat{K}_x 
          & \hat{B}^\top \hat{A}^{-\top} \emtx
         \bmtx \hat{x}_t \\ \bar{x}_t \emtx  \\
      & \hspace{1.5in}
      + (I-\hat{B}^\top \hat{A}^{-\top} \hat{K}_x^\top)\hat{H} \, d_t. 
  \end{align*}
}

Next, apply the following coordinate transformation:
\begin{align*}
  \bmtx \tilde{x}_t \\ \bar{x}_t \emtx =
  \bmtx I & \hat{Y} \\ 0 & I \emtx
         \bmtx \hat{x}_t \\ \bar{x}_t \emtx :=
  T \, \bmtx \hat{x}_t \\ \bar{x}_t \emtx. 
\end{align*}
Note that $T^{-1} = \bsmtx I & -\hat{Y} \\ 0 & I\esmtx$.
We can simplify the results of the coordinate transformation
by using the DARE~\eqref{eq:DAREYhat} to show:
\begin{align*}
  \hat{Y} \hat{A}^{-\top}\hat{K}_x^\top\hat{H} & = 
     \left( \hat{A}\hat{Y} 
     - \hat{A}\hat{Y} \hat{K}_x^\top \hat{W}^{-1} \hat{K}_x \hat{Y} \right)
       \hat{K}_x^\top \hat{H} \\
     & = \hat{A}\hat{Y} \hat{K}_x^\top \hat{W}^{-1}
       \left( \hat{W} - \hat{K}_x \hat{Y} \hat{K}_x^\top \right) \hat{H}
       = \hat{K}_y^\top.
\end{align*}
Thus $\hat{A} -\hat{Y} \hat{A}^{-\top}\hat{K}_x^\top\hat{H}\hat{K}_x$
is equal to $A_F = \hat{A}-\hat{K}_y^\top \hat{K}_x$.  Similarly, we
can use the DARE~\eqref{eq:DAREYhat} to show
$\hat{Y}\hat{A}^{-\top} = A_F\hat{Y}$.  Apply the transformation $T$
to the state matrix of $\hat{F}^\sim \hat{F}$ using these relations:
\begin{align*}
  &  T \bsmtx \hat{A} & 0  \\ 
  -\hat{A}^{-\top}\hat{K}_x^\top\hat{H}\hat{K}_x  
  & \hat{A}^{-\top} \esmtx T^{-1} 
    =   \bsmtx A_F        & A_F \hat{Y}   \\ 
  -\hat{A}^{-\top}\hat{K}_x^\top\hat{H}\hat{K}_x  
  & \hat{A}^{-\top} \esmtx T^{-1} \\
  & \hspace{0.5in}  =  \bsmtx A_F     & 0 \\
  -\hat{A}^{-\top}\hat{K}_x^\top\hat{H}\hat{K}_x \,\,
  &  \hat{A}^{-\top} (I +\hat{K}_x^\top\hat{H}\hat{K}_x \hat{Y})   \esmtx.
\end{align*}
The $(2,2)$ block can be simplified using
$A_F=\hat{A} -\hat{Y} \hat{A}^{-\top}\hat{K}_x^\top\hat{H}\hat{K}_x$
and $\hat{Y}\hat{A}^{-\top} = A_F\hat{Y}$. Combining these two
expressions gives
$A_F (I +\hat{Y} \hat{K}_x^\top\hat{H}\hat{K}_x ) = \hat{A}$.  Hence
the lower right block is $A_F^{-\top}$. This can also be used to
simplify the $(1,2)$ block as follows:
{\footnotesize
\begin{align*}
  & -\hat{A}^{-\top}\hat{K}_x^\top\hat{H}\hat{K}_x = 
    -A_F^{-\top} (I +\hat{K}_x^\top\hat{H}\hat{K}_x \hat{Y})^{-1}
    \hat{K}_x^\top\hat{H}\hat{K}_x  \\
  &  = -A_F^{-\top} (I -\hat{K}_x^\top \hat{W}^{-1} \hat{K}_x \hat{Y} )
    \hat{K}_x^\top\hat{H}\hat{K}_x 
    = -A_F^{-\top}\hat{K}_x^\top\hat{W}^{-1}\hat{K}_x.
\end{align*}
}

\noindent
The second equality makes use of the matrix inversion lemma.

  % The $(2,1)$ and $(2,2)$ blocks can be simplified to
  % $-A_F^{-\top}\hat{K}_x^\top\hat{W}^{-1}\hat{K}_x$ and $A_F^{-\top}$,
  % respectively.

Similar simplifications can be made after applying $T$ to the
remaining state matrices of $\hat{F}^\sim \hat{F}$. Thus the
coordinate transformation gives the following realization for
$\hat{F}^\sim\hat{F}$:
{\footnotesize
  \begin{align*}
    \bmtx \tilde{x}_{t+1} \\ \bar{x}_{t+1} \emtx & = 
    \bmtx A_F & 0  \\ -A_F^{-\top}\hat{K}_x^\top\hat{W}^{-1}\hat{K}_x 
     & A_F^{-\top} \emtx \, \bmtx \hat{x}_t \\ \bar{x}_t \emtx +
    \bmtx B_F \\ -A_F^{-\top}\hat{K}_x^\top\hat{W}^{-1} \emtx d_t \\
    \bar{d}_t & 
     = \bmtx (I-B_F^\top A_F^{-\top} \hat{K}_x^\top)\hat{W}^{-1}\hat{K}_x 
           & B_F^\top A_F^{-\top} \emtx
          \bmtx \hat{x}_t \\ \bar{x}_t \emtx  \\
       & \hspace{1in}
       + (I-B_F^\top A_F^{-\top} \hat{K}_x^\top)\hat{W}^{-1} \, d_t. 
  \end{align*}
}It can be shown that this is also a realization for $F^\sim F$ where
$F$ is the system in Equation~\eqref{eq:FSF}. Hence we have shown that
$\hat{P}^\sim \hat{P} = \hat{F}^\sim \hat{F} = F^\sim F$.  

The proof is concluded by showing that $F$ satisfies the other
properties required of a spectral factor: $F$ is square,
$n_d\times n_d$ system that is stable/causal and with a stable/causal
inverse. The input and output dimensions of $F$ are equal to $n_d$ so
$F$ is square. Moreover, $A_F := \hat{A}-\hat{K}_y^\top\hat{K}_x$ is a
Schur matrix by Statement 2. Hence $F$ is stable and causal. Next,
$F^{-1}$ exists because $D_F$ is invertible and $F^{-1}$ is given by:
\begin{align}
  \begin{split}
  \tilde{x}_{t+1} & = (A_F-B_F D_F^{-1} C_F) \, \tilde{x}_t  
                  + (B_F D_F^{-1}) \, \tilde{e}_t \\
   d_t  & = -(D_F^{-1} C_F) \, \tilde{x}_t   
              +  D_F^{-1}  \, \tilde{e}_t.
\end{split}
\end{align}
The poles of $F^{-1}$ are given by the eigenvalues of:
\begin{align}
  A_F - B_F D_F^{-1} C_F = \hat{A}-\hat{B}\hat{K}_x.
\end{align}
Thus $A_F - B_F D_F^{-1} C_F$ is a Schur matrix by Statement 1
and hence $F^{-1}$ is stable/causal.   Thus $F$ is
a spectral factor of $\hat{P}^\sim \hat{P}$ as claimed.
\end{proof}

% The results use certain nonsingularity assumptions (given below) to
% avoid the need for descriptor systems.  These nonsingularity
% assumptions can be removed using more general spectral factorization
% results for descriptor systems \cite{katayama94}.

\subsection{Spectral Factorization for Regret Bound}
\label{app:SFforRegret}

This appendix gives a spectral factorization result tailored
for the regret bound in Section~\ref{sec:nomfeas}.  The
closed-loop dynamics with $K^o$, given in \eqref{eq:CLK0}, have the
form:
\begin{align}
\tag{\ref{eq:CLK0}, Revisited}
\begin{split}
\hat{x}_{t+1} & = \hat{A} \, \hat{x}_t + \hat{B} \, d_t \\
  \hat{e}_t & = \hat{C}\, \hat{x}_t + \hat{D} \, d_t,
\end{split}
\end{align}
with the state matrices:
\begin{align*}
&\hat{A}:=\bmtx \hat{A}_{11} & -B_u K_v\hat{A}_{11}^{-\top}
\\ 0 & \hat{A}_{11}^{-\top} \emtx,
&& \hspace{-0.1in} \hat{B}:= \bmtx B_d \\ -X B_d \emtx, \\
& \hat{C}:= \gamma_J \,\bmtx C_e-D_{eu}K_x & -D_{eu}K_v\hat{A}_{11}^{-\top} 
      \\ 0 & 0 \emtx,
&& \hat{D}:= \gamma_d \bmtx 0 \\ I \emtx.
\end{align*}
Here $X\succeq 0$ is the stabilizing solution of the
DARE \eqref{eq:DARE}:
\begin{align}
\tag{\ref{eq:DARE}, Revisited}
  0 & = X  - A^\top XA  - Q \\
\nonumber
    & + (A^\top XB_u+S)\, (R+B_u^\top XB_u)^{-1} \, (A^\top XB_u+S)^\top.
\end{align}
The stabilizing gain is $K_x$ and
$\hat{A}_{11} := A - B_u K_x\in \R^{n_x\times n_x}$ is a Schur,
nonsingular matrix. Let $\hat{P}$ denote these closed-loop dynamics
with state, input, and output dimension $2n_x$, $n_d$, and
$n_{\hat{e}}$.  Recall that the regret bound in
Equation~\ref{eq:RegretBnd} is $\| \hat{e}\|_2^2=\|\hat{P}d\|_2$. 

The next lemma is the main result in this appendix.  It provides a
spectral factorization for $\hat{P}^\sim \hat{P}$. This is used in the
main body of the text to rewrite the non-causal cost in terms of the
spectral factor.

% XXX
%\setcounter{lemma}{\ref{lem:RegretSF}}

\begin{lemma}
\label{lem:RegretSFVer2}
Assume that $(A,B_u,C_e,D_{eu})$ satisfy conditions $(i)-(iv)$ in
Lemma~\ref{lem:DARE} so that the DARE \eqref{eq:DARE} has a
stabilizing solution $X\succeq 0$. This stabilizing solution is used
to define the dynamics $\hat{P}$ from $d_t$ to $\hat{e}_t$ in
\eqref{eq:CLK0}.

If $\gamma_d>0$ and $(\hat{A}_{11}^{-\top},X B_d)$ is stabilizable
then there exists a square $n_d\times n_d$ LTI system $F$ such that:
(i) $\| \hat{e}\|_2^2 = \| F d\|_2^2$ for all $d \in \ell_2$, (ii) $F$
is stable, causal, and invertible, and (iii) $F^{-1}$ is square,
stable and causal.  Moreover, it is possible to construct $F$ with
state dimension $n_x$.
\end{lemma}
\begin{proof}
  Consider the dynamics in Equation~\ref{eq:CLK0} and define
  $\hat{Q}:=\hat{C}^\top \hat{C}$, $\hat{S}:= \hat{C}^\top \hat{D}=0$,
  and $\hat{R} = \hat{D}^\top \hat{D}$. Note that $\hat{S}=0$ and
  $\hat{R} = \gamma_d^2 I$. These matrices satisfy the following:

  \vspace{0.1in}
  \noindent (i) $\hat{R} = \gamma_d^2 I \succ 0$ by the assumption
  that $\gamma_d>0$.

  \noindent (ii) $(\hat{A},\hat{B})$ is stabilizable
  because $(\hat{A}_{11}^{-\top},\, -XB_d)$ is stabilizable (by
  assumption) and $\hat{A}_{11}$ is a Schur matrix.

  \noindent (v,vi) $\hat{A}$ is nonsingular and has no eigenvalues on
  the unit circle.  This follows because the eigenvalues of $\hat{A}$
  are the union of the eigenvalues of $\hat{A}_{11}$ and
  $\hat{A}_{11}^{-\top}$. Hence stability and nonsingularity of
  $\hat{A}_{11}$ (Lemma~\ref{lem:DARE}) implies that $\hat{A}$ is
  nonsingular and has no eigenvalues on the unit circle.  

  \noindent (iii) $\hat{A}-\hat{B}\hat{R}^{-1}\hat{S}^\top$ is
  nonsingular because $\hat{S}=0$ and $\hat{A}$ is nonsingular.

  \noindent (iv)
  $\bsmtx \hat{A}-e^{j\theta} I & \hat{B} \\ \hat{C} & \hat{D} \esmtx$
  has full column rank for all $\theta \in [0,2\pi]$. Based on the
  structure of the matrices, this follows because
  $(\hat{A}_{11}^{-\top},\, -XB_d)$ is stabilizable (by assumption)
  and $\hat{A}_{11}$ is a Schur matrix.
  \vspace{0.1in}

  In summary, the matrices satisfy assumptions (i)-(vi) in
  Lemma~\ref{lem:DTSF}. This ensures that $\hat{P}^\sim \hat{P}$ has a
  spectral factor $F$.
  % The existence of a spectral factor follows from
  % Lemma~\ref{lem:DTSF}.

  The spectral factor in Lemma~\ref{lem:DTSF} has the same state
  dimension as $\hat{A}$ which is $2n_x$. The remainder of the proof
  shows that a minimal realization for the spectral factor can be
  constructed with dimension $n_x$.  The key point is to verify that
  the unique solution to the DARE \eqref{eq:DAREXhat} in Statement 1
  of Lemma~\ref{lem:DTSF} has the following special structure:  
  \begin{align}
    \label{eq:XhatSpecial}
    \hat{X} = \gamma_J^2 \bmtx X & I \\ I & X^{-1} \emtx 
            + \bmtx 0 & 0 \\ 0 & V  \emtx.
  \end{align}
  This form assumes that $X$ is nonsingular which will be verified
  below. The variable $V$ is defined to be the unique positive
  semidefinite solution to the following DARE:
  \begin{align}
    \label{eq:DAREV}
    0 & = V - \hat{A}_{11}^{-1} V \hat{A}_{11}^{-\top}  - \tilde{Q} \\
    \nonumber
      & + (\hat{A}_{11}^{-1} V X B_d) \,
          \left( \gamma_d^2 I + B_d^\top X V X B_d \right)^{-1}
          \,  (\hat{A}_{11}^{-1} V X B_d)^\top,
  \end{align}
  where:
  {\footnotesize
  \begin{align*}
    \tilde{Q} := \gamma_J^2 \hat{A}_{11}^{-1} \left( X^{-1} 
        -\hat{A}_{11} X^{-1} \hat{A}_{11}^\top  
        -  B_u ( R+B_u^\top X B_u)^{-1} B_u^\top  \right) \hat{A}_{11}^{-\top}.
  \end{align*}
}

It can be shown that $\tilde{Q}$ is positive semidefinite using the
DARE for $X$. Thus Equation~\ref{eq:DAREV} for $V$ corresponds to the
DARE in Lemma~\ref{lem:DARE} with $(A,B_u,C_e,D_{eu},Q,R,S)$ replaced
by $(\hat{A}^{-\top},XB_d$,
$\bsmtx \tilde{Q}^{\frac{1}{2}} \\ 0 \esmtx ,\bsmtx 0 \\ \gamma_d
I\esmtx, \tilde{Q},\gamma_d^2 I,0)$.  It can be verified that the DARE
\eqref{eq:DAREV} satisfies conditions (i)-(iv) in Lemma~\ref{lem:DARE}
and hence there exists a stabilizing stabilizing solution
$V\succeq 0$.  It can also be shown, by direct substitution, that
$\hat{X}$ given in \eqref{eq:XhatSpecial} is positive semidefinite and
satisfies the DARE \eqref{eq:DAREXhat}. Hence this choice for
$\hat{X}$ is the unique stabilizing solution.

The stabilizing gain resulting from the special structure of
$\hat{A}$, $\hat{B}$, and $\hat{X}$ is:
\begin{align*}
  \hat{K}_x:=\hat{H}^{-1}\hat{B}^\top \hat{X} \hat{A}
  = \bmtx 0 & -\hat{H}^{-1} B_d^\top X V  \hat{A}_{11}^{-\top} \emtx.
\end{align*}
As a result, Statement 3 of Lemma~\ref{lem:DTSF} gives a spectral
factorization with the form:
\begin{align*}
  A_F & := \hat{A}-\hat{K}_y^\top\hat{K}_x
        \bmtx \hat{A}_{11} & A_{F,12} \\ 0  & A_{F,22} \emtx, \\
  C_F & := \hat{W}^{-\frac{1}{2}} \hat{K}_x 
        = \bmtx 0 &  C_{F,2} \emtx.
\end{align*}
The first block of states are unobservable in the output and hence
can be removed. Thus the spectral factorization has a minimal
realization with state dimension $n_x$ as claimed.

Finally, we will verify the assumption above that $X$ is
nonsingular. Assume, to the contrary, that $X$ is singular, i.e. there
is a nonzero vector $v$ such that $Xv=0$.  Re-write the DARE
\eqref{eq:DARE} as:
\begin{align}
  \nonumber
  X & = (A-B_uK_x)^\top X (A-B_uK_x) + (Q-S R^{-1}S^\top) \\
  \label{eq:DAREalt}
    & \hspace{0.2in}  + (K_x-R^{-1}S^\top )^\top R (K_x-R^{-1}S^\top ).
\end{align}
We have $R\succ 0$ and
$\bsmtx Q & S \\ S^\top & R \esmtx = \bsmtx C_e^\top \\ D_{eu}^\top
\esmtx \bsmtx C_e & D_{eu} \esmtx \succeq 0$. Thus
$Q-S R^{-1}S^\top \succeq 0$ by the Schur complement lemma.  Multiply
\eqref{eq:DAREalt} on the left and right by $v^\top$ and $v$.  The
left side is $v^\top X v=0$ while all three terms on the right are
$\ge 0$.  Hence each term on the right is zero so that
$X (A-B_u K_x)v=0$.

In summary, $Xv=0$ implies $X (A-B_u K_x)v=0$, i.e.  the null space of
$X$ is $(A-B_uK_x)$-invariant. It follows that there exists an
eigenvector of $(A-B_uK_x)$ in the null space of $X$ (Proposition 4.3
of \cite{demmel97}). In other words, there is a nonzero vector $v$
and scalar $\lambda$ such that $Xv=0$ and $(A-B_uK_x)v=\lambda v$.
The eigenvalue satisfies $1<|\frac{1}{\lambda}|<\infty$ since
$A-B_uK_x$ is stable and nonsingular.  This gives:
\begin{align*}
  v^\top \bmtx (A-B_uK_x)^{-\top} - \frac{1}{\lambda} I,     
  & -X B_d \emtx = 0.
\end{align*}
Thus $X$ singular implies $( (A-B_uK_x)^{-\top}, -X B_d)$ is not
stabilizable by the PBH test (Theorem 3.2 in
\cite{zhou96}).\footnote{Theorem 3.2 in \cite{zhou96} gives the PBH
  test for stabilizability of a continuous-time LTI system.  The
  discrete-time PBH test is analogous.}  By contradiction, the
assumption that $( (A-B_uK_x)^{-\top}, X )$ is stabilizable implies
that $X\succ 0$.
\end{proof}

\subsection{Robust Performance}
\label{app:RP}

This section summarizes existing results on robust performance.
These are special cases of results for the structured
singular value, $\mu$.  Details, including more general results,
can be found in \cite{packard93} or Chapter 11 of \cite{zhou96}.
All systems in this appendix are assumed to be causal.

Consider the uncertain interconnection shown in
Figure~\ref{fig:FUMDelta}. The interconnection, denoted
$F_U(M,\Delta)$, consists of an uncertainty $\Delta$ around the upper
channels of a $M$.  We will initially consider the case where
$w \in \C^{n_w}$, $d \in \C^{n_d}$, $v \in \C^{n_v}$, and
$e \in \C^{n_e}$ are complex vectors.  Moreover,
$M\in \C^{(n_v+n_e)\times (n_w+n_d)}$ and the uncertainty
$\Delta \in \C^{n_w \times n_v}$ are complex matrices. The
interconnection in Figure~\ref{fig:FUMDelta} corresponds to the
following algebraic equations:
\begin{align}
  \label{eq:MDeltaMatrix}
  \begin{split}
  \bmtx v \\ e \emtx
   & = \bmtx M_{11} & M_{12} \\ M_{21} & M_{22} \emtx
  \bmtx w \\ d \emtx \\
  w & = \Delta v,
  \end{split}
\end{align}
where the matrix $M$ has been block partitioned conformably with the
input/output signals.  

\begin{figure}[h]
\centering
\scalebox{0.9}{
\begin{picture}(140,90)(23,65)
 \thicklines
% I/O for M
 \put(75,65){\framebox(40,40){$M$}}
 \put(160,72){$d$}
 \put(155,75){\vector(-1,0){40}}  
 \put(25,72){$e$}
 \put(75,75){\vector(-1,0){40}}  
% I/O for Delta
 \put(80,115){\framebox(30,30){$\Delta$}}
 \put(43,95){$v$}
 \put(55,95){\line(1,0){20}}  
 \put(55,95){\line(0,1){35}}  
 \put(55,130){\vector(1,0){25}}  
 \put(141,95){$w$}
 \put(135,130){\line(-1,0){25}}  
 \put(135,130){\line(0,-1){35}}  
 \put(135,95){\vector(-1,0){20}}  
\end{picture}
} % End scalebox
\caption{Uncertain interconnection $F_U(M,\Delta)$.}
\label{fig:FUMDelta}
\end{figure}
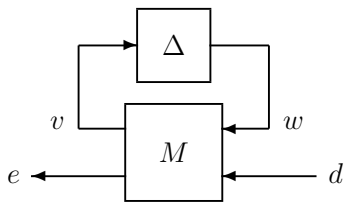

The matrix interconnection $F_U(M,\Delta)$ is said to be well-posed if
$I_{n_v}-M_{11}\Delta$ is nonsingular.  If the interconnection is
well posed then for each $d\in \C^{n_d}$ there exist unique $(v,e,w)$
that satisfy \eqref{eq:MDeltaMatrix}.  Moreover, the output $e$ is
given by:
\begin{align*}
  e  = [ 
\underbrace{M_{22} + M_{21} \Delta (I_{n_v}-M_{11} \Delta)^{-1} M_{12}}_{:=F_U(M,\Delta)}
           ] \, d . 
%  e & = \left[ M_{22} + M_{21} \Delta (I_{n_v}-M_{11} \Delta)^{-1} M_{12} 
%           \right] \, d  \\
%   & := F_U(M,\Delta) \, d
\end{align*}
The first main result concerns the well-posedness and gain of
$F_U(M,\Delta)$ when the uncertainty satisfies
$\bar{\sigma}(\Delta) \le 1$.  The proof uses the equivalence of
$\|\cdot \|_{2\to 2}$ and $\bar{\sigma}( \cdot )$ for matrices.

\begin{lemma}
  \label{lem:MatrixSSV}
  Let $M\in \C^{(n_v+n_e)\times (n_w+n_d)}$ be given.  The following
  are equivalent:
  \begin{enumerate}
    \renewcommand{\theenumi}{\Alph{enumi}}
  \item There exists $D\in \R$ with $D>0$ such that:
    \begin{align}
      \bar{\sigma} \left( \bsmtx D \cdot I_{n_v} & 0 \\ 0 & I_{n_e} \esmtx
        M \bsmtx D^{-1} \cdot I_{n_w} & 0 \\ 0 & I_{n_d} \esmtx
      \right) < 1.
    \end{align}
  \item If $\Delta \in \C^{n_w \times n_v}$ and
    $\Delta_P \in \C^{n_d\times n_e}$ satisfy
    \begin{align}
      \label{eq:detMDelta}
      \det\left( I-M \bmtx \Delta & 0 \\ 0 & \Delta_P \emtx \right) = 0.
    \end{align}                                         
    then
    $\max\left( \bar{\sigma}(\Delta), \bar{\sigma}(\Delta_P) \right) >1$.
  % \item Assume $\Delta \in \C^{n_w \times n_v}$ satisfies
  %   $\bar{\sigma}(\Delta)\le 1$. Then $F_U(M,\Delta)$ is well-posed
  %   and $\|F_U(M,\Delta)\|_{2\to 2} < 1$.
  \item $F_U(M,\Delta)$ is well-posed and
    $\bar{\sigma}(F_U(M,\Delta)) < 1$ for all $\Delta \in \C^{n_w \times n_v}$ 
    with  $\bar{\sigma}(\Delta)\le 1$. 
\end{enumerate}
\end{lemma}
\begin{proof}
  Equivalence of A) and B) follows from the fact that the structured
  singular value ($\mu$) is equal to its upper bound for the case of
  two full, complex blocks. This was originally shown in
  \cite{redheffer59} but a more recent reference is Theorem 9.1 of
  \cite{packard93}.  Equivalence of B) and C) is a special case of the
  main-loop theorem (Theorem 4.3 of \cite{packard93}). 

  We'll sketch the proof of sufficiency, i.e. A)$\to $ B) $\to$
  C). This direction is used to verify the bound on robust regret and
  can be shown with only linear algebra facts.

\vspace{0.1in}
\noindent
\emph{A)$\to$ B):} Assume A) is true and that $\Delta$, $\Delta_P$
satisfy \eqref{eq:detMDelta}. We will show that
$\max\left( \bar{\sigma}(\Delta), \bar{\sigma}(\Delta_P) \right) >1$,
i.e. B) holds.

Note that $\Delta = D^{-1} \Delta D$
because $D>0$ is a scalar. Use this relation and the Sylvester
determinant identity to re-write \eqref{eq:detMDelta} as:
\begin{align*}
  \det \Bigl( I-
\underbrace{\bsmtx D \cdot I_{n_v} & 0 \\ 0 & I_{n_e} \esmtx
M \bsmtx D^{-1} \cdot I_{n_w} & 0 \\ 0 & I_{n_d} \esmtx}_{:=\hat{M}}
\underbrace{\bsmtx \Delta & 0 \\ 0 & \Delta_P \esmtx}_{:=\hat{\Delta}} 
\, \Bigr) = 0.
\end{align*}  
Thus $I-\hat{M}\hat{\Delta}$ is singular so there is a nonzero
vector $v$ such that $\hat{M}\hat{\Delta}v = v$.  Hence
$\|\hat{M}\hat{\Delta}\|_{2\to 2} \ge
\frac{\|\hat{M}\hat{\Delta}v\|_2}{\|v\|_2} = 1$.  Further, the induced
2-norm is sub-multiplicative:
\begin{align*}
\|\hat{M}\hat{\Delta}\|_{2\to 2} \le \|\hat{M}\|_{2\to 2} \cdot
\|\hat{\Delta}\|_{2\to 2}. 
\end{align*}
Thus $\|\hat{M}\hat{\Delta}\|_{2\to 2} \ge 1$ implies
$\|\hat{\Delta}\|_{2\to 2} \ge \|\hat{M}\|_{2\to 2}^{-1}$. Finally,
combine this with $\|\hat{M}\|_{2\to 2}<1$, by assumption A), to
conclude that
$\bar{\sigma}(\hat{\Delta})=\|\hat{\Delta}\|_{2\to 2}>1$.  Condition
B) follows because
$\max\left( \bar{\sigma}(\Delta), \bar{\sigma}(\Delta_P) \right)
=\bar{\sigma}( \hat\Delta )>1$.

%$\max\left( \bar{\sigma}(\Delta), \bar{\sigma}(\Delta_P) \right)
%=\bar{\sigma}\left( \bsmtx \Delta & 0 \\ 0 & \Delta_P \esmtx \right)>1$.

% It follows that
% $\bar{\sigma} \left( \bsmtx \Delta & 0 \\ 0 & \Delta_P \esmtx \right)
% > 1$.\footnote{ Let $A$, $B$ be matrices that satisfy $\det(I-AB)=0$
%   with $\|A\|_{2\to 2}<1$. This implies there is a nonzero vector $v$
%   such that $ABv = v$.  Hence
%   $\|AB\|_{2\to 2} \ge \frac{\|ABv\|_2}{\|v\|_2} = 1$.  Further, the
%   induced 2-norm is sub-multiplicative:
%   $\|AB\|_{2\to 2} \le \|A\|_{2\to 2} \cdot \|B\|_{2\to 2}$. Thus
%   $\|AB\|_{2\to 2} \ge 1$ implies
%   $\|B\|_{2\to 2} \ge \|A\|_{2\to 2}^{-1}$. Finally, combine this with
%   $\|A\|_{2\to 2}<1$ to conclude that
%   $\bar{\sigma}(B)=\|B\|_{2\to 2}>1$.} Condition B) follows because
% $\bar{\sigma}\left( \bsmtx \Delta & 0 \\ 0 & \Delta_P \esmtx \right)
% =\max\left( \bar{\sigma}(\Delta), \bar{\sigma}(\Delta_P) \right)$.

\vspace{0.1in}
\noindent
\emph{B)$\to$ C):} Assume B) is true and that $\Delta$ satisfies
$\bar{\sigma}(\Delta) \le 1$. We will show that $F_U(M,\Delta)$ is
well-posed and $\bar{\sigma}(F_U(M,\Delta))<1$, i.e. C) holds.

Set $\Delta_P=0$ so that
$\max\left( \bar{\sigma}(\Delta), \bar{\sigma}(\Delta_P) \right) \le
1$.  It follows from B) that:
\begin{align*}
  0  \ne \det\left( I-M \bmtx \Delta & 0 \\ 0 & 0 \emtx \right)
     = \det(I_{n_v}- M_{11} \Delta ).
\end{align*}                                         
Hence $F_U(M,\Delta)$ is well-posed. 

Next, we will show $\gamma:=\bar{\sigma}(F_U(M,\Delta))<1$. This holds
trivially if $\gamma=0$ so we'll assume that $\gamma>0$. There exist
vectors $d\in\C^{n_d}$ and $e\in \C^{n_e}$ such that
$e = F_U(M,\Delta) d$, $\|d\|_2 =1$, and $\|e\|_2 = \gamma$.  Define
$\Delta_P:= \gamma^{-2} d e^*$ so that $\Delta_P e = d$ and
$\bar{\sigma}(\Delta_P) = \gamma^{-1}$. By construction,
$F_U(M,\Delta) \Delta_P \, e = e$ and hence $I-F_U(M,\Delta) \Delta_P$
is singular.

We have $\det(I-M_{11}\Delta)\ne 0$ by well-posedness.  Hence singularity
of $I-F_U(M,\Delta) \Delta_P$ and the definition of
$F_U(M,\Delta)$ imply: {\footnotesize
  \begin{align*}
    & 0  = \det(I-M_{11}\Delta) \cdot
    \det( I - F_U(M,\Delta) \Delta_P )
    \\
    & = \det(I-M_{11}\Delta) \cdot
    \det( I - M_{22}\Delta_P - M_{21}\Delta (I-M_{11}\Delta)^{-1} M_{12}\Delta_P). 
  \end{align*}
}This can be rewritten, by the block matrix determinant formula, as:
\begin{align*}
  0 = \det \bsmtx I-M_{11}\Delta & -M_{12} \Delta_P \\ -M_{21}\Delta &
       I - M_{22}\Delta_P \esmtx
    = \det \left( I - M \bsmtx \Delta & 0 \\ 0 & \Delta_P \esmtx \right).
\end{align*}
Finally, assumption B) implies that
$\max\left( \bar{\sigma}(\Delta), \bar{\sigma}(\Delta_P) \right)>1$.
Moreover,  $\bar{\sigma}(\Delta) \le 1$, and
$\bar{\sigma}(\Delta_P)=\gamma^{-1}$ give:
\begin{align*}
\gamma^{-1} =  \max\left( \bar{\sigma}(\Delta), \bar{\sigma}(\Delta_P) \right)>1.
\end{align*}
Hence $\bar{\sigma}(F_U(M,\Delta))=\gamma<1$.
\end{proof}

% Moreover, if condition A) fails to hold then there exists a
% $\Delta \in \C^{n_w \times n_v}$ such that: (i)
% $\bar{\sigma}(\Delta)\le 1$ and (ii) either $F_U(M,\Delta)$ is not
% well-posed or $\|F_U(M,\Delta)\| \ge 1$.  This follows by
% contraposition of condition C).

% If $M \in \R^{n_w \times n_v}$ and A) fails to hold then we can
% construct a $\Delta \in \R^{n_w \times n_v}$ such that (i) and (ii)
% hold, i.e. the uncertainty can be constructed to be real rather than
% complex.  This follows from results for the structured singular value
% with two full blocks (Section 9.7 of \cite{packard93}).
%
% In either case, $\Delta$ can be constructed as a rank 1 matrix (Remark
% 3.3 of \cite{packard93}).

Next, consider the uncertain interconnection in
Figure~\ref{fig:FUMDelta} where $M$ and $\Delta$ are discrete-time
systems. Assume $M$ has the following state-space representation:
\begin{align}
  \label{eq:MSS}
   \bmtx x_{t+1}\\ v_t \\ e_t\emtx =
   \bmtx A & B_w & B_d\\
   C_v & D_{vw} & D_{vd} \\
   C_e & D_{ew} & D_{ed}  \\
   \emtx \bmtx x_t \\ w_t \\d_t \emtx.
\end{align}
Moreover, assume $\Delta$ has the state-space representation:
\begin{align}
  \label{eq:DeltaSS}
   \bmtx x^\Delta_{t+1}\\ w_t  \emtx =
   \bmtx A_\Delta & B_\Delta \\
      C_\Delta & D_\Delta \emtx
     \bmtx x^\Delta_t \\ v_t \emtx.
\end{align}
Combine the $v_t$ and $w_t$ output equations of \eqref{eq:MSS} and
\eqref{eq:DeltaSS}:
\begin{align}
\label{eq:vwWellPosed}
  \bmtx I_{n_v} & -D_{vw} \\ -D_\Delta & I_{n_w} \emtx \, 
  \bmtx v_t \\ w_t \emtx  =
  \bmtx C_v & 0 \\ 0 & C_\Delta \emtx \, \bmtx x_t \\ x^\Delta_t \emtx
  +  \bmtx D_{vd} \\ 0  \emtx \, d_t.
\end{align}
We say that (system) interconnection $F_U(M,\Delta)$ is well-posed if
$\bsmtx I_{n_v} & -D_{vw} \\ -D_\Delta & I_{n_w} \esmtx$ is
nonsingular.  This condition is equivalent to nonsingularity of
$I_{n_v} - D_{vw} D_\Delta$. Hence if the system is well-posed then
\eqref{eq:vwWellPosed} has a unique solution $(v_t,w_t)$ for any
$(x_t,x^\Delta_t,d_t)$.  In addition, $F_U(M,\Delta)$ has a well-defined
state-space representation for the dynamics from $d$ to $e$:
\begin{align*}
  \bmtx x_{t+1} \\ x^\Delta_{t+1} \emtx  & = \tilde{A} \, 
    \bmtx x_t \\ x^\Delta_t \emtx  + \tilde{B} \, d_t \\
  e_t & = \tilde{C} \bmtx x_{t} \\ x^\Delta_{t} \emtx + \tilde{D} \, d_t,
\end{align*}
where:
\begin{align}
  \label{eq:MDeltaAMatrix}
  \tilde{A}:= \bsmtx A & 0 \\ 0 & A_\Delta \esmtx 
+ \bsmtx 0 & B_w \\ B_\Delta & 0 \esmtx \,                                
\bsmtx I_{n_v} & -D_{vw} \\ -D_\Delta & I_{n_w} \esmtx^{-1}
\, \bsmtx C_v & 0 \\ 0 & C_\Delta \esmtx, 
\end{align}
and $(\hat{B},\hat{C},\hat{D})$ can be defined similarly.

If the interconnection is well posed then for each $d\in \ell_2$ there
exist unique signals $(v,e,w,x,x^\Delta)$ that satisfy the
$(M,\Delta)$ dynamics from $x_{-\infty}=0$ and $x^\Delta_{-\infty}=0$.
The small gain theorem, stated next, provides a necessary and
sufficient condition for well-posedness and stability of the system
$F_U(M,\Delta)$ when the uncertainty satisfies
$\| \Delta \|_\infty \le 1$. The condition is stated in terms of
$(1,1)$ block of $M$, denoted $M_{11}$. This corresponds to the system
with input $w$ to output $v$ defined by state matrices
$(A,B_w,C_v,D_{vw})$.
\begin{thm}
  \label{thm:smallgain}
  Let $M$ be a given $(n_v+n_e)\times(n_w+n_d)$ LTI system and assume
  $M$ is stable.  The following are equivalent:
  \begin{enumerate}
    \renewcommand{\theenumi}{\Alph{enumi}}
  \item  $\|M_{11}\|_\infty <1$. 
  \item $F_U(M,\Delta)$ is well-posed and stable for all
    $n_w\times n_v$ stable, LTI systems $\Delta$ with
    $\|\Delta\|_\infty \le 1$.
  \end{enumerate}
\end{thm}
\begin{proof}
  The sufficiency of the small gain condition under more general
  settings is due to Zames \cite{khalil5,zames66}. Condition A) is
  actually necessary and sufficient for LTI systems and uncertainties.
  Theorem 9.1 in \cite{zhou96} provides a proof for the
  continuous-time version of this result.  We sketch a proof for
  discrete-time systems for completeness.  

  A key step is to connect the eigenvalues of $\tilde{A}$
  to the transfer function $M_{11}(z)\Delta(z)$.  Use the
  block-matrix determinant formula and definition of the transfer
  functions  to show:
  \begin{align*}
    &\det(I_{n_v} - M_{11}(z) \Delta(z)) \\
    & = \det\left( \bmtx I_{n_v} & 0 \\ 0 & I_{n_w} \emtx
          - \bmtx 0 & M_{11}(z) \\ \Delta(z) & 0 \emtx \right) \\
 & = \det\left( \bsmtx I_{n_v} & -D_{vw} \\ -D_\Delta  & I_{n_w} \esmtx
     - \bsmtx C_v & 0 \\ 0 & C_\Delta \esmtx
      \bsmtx zI-A & 0 \\ 0 & zI-A_\Delta\esmtx^{-1}
      \bsmtx 0 & B_w \\ B_\Delta & 0 \esmtx
                              \right). 
  \end{align*}
  Apply the block-matrix determinant formula again to the last line to obtain:
  \begin{align}
    \label{eq:SGdetcond}
    \det(I_{n_v} - M_{11}(z) \Delta(z))  = 
      \det( zI-\tilde{A}).
  \end{align}
  This equality holds when: (i) $F_U(M,\Delta)$ is well-posed and (ii) $z$ is
  not a pole of $M_{11}$ or $\Delta$.  These conditions ensure that
  the various determinants are well-defined.  We can now complete
  the proof.

\vspace{0.1in}
\noindent
\emph{A)$\to$ B):} Assume A) holds and $\Delta$ is any stable system
with $\|\Delta\|_\infty \le 1$. We will show that B) holds:
$F_U(M,\Delta)$ is well-posed and stable.

Define the complement of the open unit disk as
$\Dc:=\{ z \, : \, |z|\ge 1 \}$.  Both $M_{11}$ and $\Delta$ are
stable systems so the maximum modulus principle holds \cite{boyd85}:
\begin{align*}
  \sup_{z \in \Dc} \bar{\sigma}\left( M_{11}(z) \Delta(z) \right) 
  = \sup_{\theta \in [0,2\pi]}
  \bar{\sigma}\left( M_{11}(e^{j\theta}) \Delta(e^{j\theta}) \right).
\end{align*}
The right-side is equal to $\|M_{11} \Delta\|_\infty$. This can be bounded
using the small gain assumptions:
\begin{align}
  \| M_{11} \Delta\|_\infty 
    \le   \| M_{11} \|_\infty \cdot   \| \Delta\|_\infty < 1.
\end{align}
Hence $\bar{\sigma}\left( M_{11}(z) \Delta(z) \right) <1$ for all
$z\in \Dc$. This further implies that
that $\|M_{11}(z)\Delta(z)v\|_2 < \|v\|_2$ for any nonzero
vector $v$.  Hence $(I-M_{11}(z)\Delta(z))v \ne 0$ for any
$v\ne 0$. Therefore $\det(I_{n_v} - M_{11}(z) \Delta(z))\ne 0$ for all
$z\in \Dc$.

One consequence is that $I_{n_v}-D_{vw} D_\Delta$ is nonsingular
because $M_{11}(\infty)=D_{vw}$  and $\Delta(\infty) = D_\Delta$.
Thus $F_U(M,\Delta)$ is well-posed. Moreover, any $z\in \Dc$ is not a pole
of $M_{11}$ and $\Delta$ because both are stable. Hence we
can apply \eqref{eq:SGdetcond}:
\begin{align*}
 0 \ne \det(I_{n_v} - M_{11}(z) \Delta(z))  = 
      \det( zI-\tilde{A})
\,\,\, \forall z\in \Dc.
\end{align*}
Thus $F_U(M\Delta)$ is stable.

\vspace{0.1in}
\noindent
\emph{B)$\to$ A):} Assume A) fails to hold and we will show that B)
fails to hold.

If A) fails to hold then there is a frequency $\theta_0\in[0,2\pi]$
such that $\bar{\sigma}(M(e^{j\theta_0}))\ge 1$. To shorten notation,
define $M_0:=M(e^{j\theta_0})$. Let $u_0\in \C^{n_v}$,
$v_0\in \C^{n_w}$ be the singular vectors corresponding to
$\bar{\sigma}(M_0)$: $\|u_0\|_2=\|v_0\|_2=1$ and
$\bar{\sigma}(M_0) u_0 =M_0 v_0$. Define
$\Delta_0 := \frac{1}{\bar{\sigma}(M_0)} v_0 u_0^*$ so that
$\bar{\sigma}(\Delta_0) = \frac{1}{\bar{\sigma}(M_0)} \le 1$.  
Moreover, $(I_{n_v}-M_0 \Delta_0) u_0 = 0$ which implies
$\det(I_{n_v}-M_0\Delta_0)=0$.

There are two cases.  First, if $\theta_0=0$, $\pi$, or $2\pi$ then
$M_0$ and $\Delta_0$ are real matrices and we can define constant
system $\Delta:=\Delta_0$.  Second, if
$\theta_0 \in (0,\pi) \cup (\pi,2\pi)$ then $M_0$ and $\Delta_0$ are
complex matrices. We can construct a stable, LTI system $\Delta$ such
that $\Delta(e^{j\theta_0})=\Delta_0$ and
$\|\Delta\|_\infty = \bar{\sigma}(\Delta_0)$.  Theorem 9.1 of
\cite{zhou96} gives the construction for continuous-time and a similar
construction works in discrete-time.  

In either case, the construction gives a stable, LTI system $\Delta$
with $\|\Delta\|_\infty =\bar{\sigma}(\Delta_0)\le 1$ and
$\det(I_{n_v} - M(e^{j\theta_0}) \Delta(e^{j\theta_0}))=0$.  It
follows from \eqref{eq:SGdetcond} that $\tilde{A}$ has a pole at
$z=e^{j\theta_0}$ and hence $F_U(M,\Delta)$ is unstable.
\end{proof}

Finally, we arrive at the main result of this appendix: a necessary
and sufficient condition for the well-posedness, stability and
gain-boundedness of the system $F_U(M,\Delta)$ when the uncertainty
satisfies $\| \Delta \|_\infty \le 1$.
\begin{lemma}
  \label{lem:SystemSSV}
  Let $M$ be a given $(n_v+n_e)\times(n_w+n_d)$ LTI system and assume
  $M$ is stable. The following are equivalent:
  \begin{enumerate}
    \renewcommand{\theenumi}{\Alph{enumi}}
  \item There exists $D:[0,2\pi]\to (0,\infty)$ such that 
    \begin{align}
      \nonumber
     & \bar{\sigma}\left( \bsmtx D(\theta) \cdot I_{n_v} & 0 \\ 0 & I_{n_e} \esmtx
    M(e^{j\theta}) \bsmtx D(\theta)^{-1} \cdot I_{n_w} & 0 \\ 0 & I_{n_d} \esmtx
   \right) < 1 \\ 
      \label{eq:DMDi}
     & \hspace{2in} \forall \theta \in [0,2\pi].
    \end{align}
  % \item Let $\Delta$ be a given $n_w\times n_v$ LTI system.  Assume
  %   $\Delta$ is stable and $\|\Delta\|_\infty \le 1$.  Then
  %   $F_U(M,\Delta)$ is well-posed, stable, and
  %   $\|F_U(M,\Delta)\|_\infty < 1$.
  \item $F_U(M,\Delta)$ is well-posed, stable, and
    $\|F_U(M,\Delta)\|_\infty < 1$ for all $n_w\times n_v$ stable, LTI
    systems $\Delta$ with $\|\Delta\|_\infty \le 1$.
\end{enumerate}
\end{lemma}
\begin{proof}
  This lemma is related to the frequency domain structured singular
  value ($\mu$) tests in Section 10.2 of \cite{packard93}. Similar
  related results appear in Chapter 11 of \cite{zhou96}.  However, the
  precise statement of this lemma does not appear in either reference.
  Hence we'll sketch the proof here.

\vspace{0.1in}
\noindent
\emph{A)$\to$ B):} Assume A) holds and $\Delta$ is stable with
$\|\Delta\|_\infty \le 1$. We will show that B) holds: $F_U(M,\Delta)$
is well-posed, stable, and $\| F_U(M,\Delta)\|_\infty <1$.

First, the upper left block of the matrix in \eqref{eq:DMDi} is
$D(\theta)M_{11}(e^{j\theta}) D(\theta)^{-1}$. This simplifies to
$M_{11}(e^{j\theta})$ because $D(\theta)$ is a nonzero scalar.
Hence Condition A) implies that $\|M_{11}\|_\infty <1$. Well-posedness
and stability of $F_U(M,\Delta)$ follows from the small-gain theorem
(Theorem~\ref{thm:smallgain}).

Next, note that Condition A) in Lemma~\ref{lem:MatrixSSV} holds for
each $\theta\in[0,2\pi]$ with the complex matrix $M(e^{j\theta})$ and
real scalar $D(\theta)>0$. This implies that Condition C) in
Lemma~\ref{lem:MatrixSSV} holds using the complex matrix
$\Delta(e^{j\theta})$, i.e:
\begin{align}
  \bar{\sigma} \left( F_U(M(e^{j\theta}),\Delta(e^{j\theta}))\right) <1 
  \,\,\, \forall \theta \in [0,2\pi].
\end{align}
Thus $\|F_U(M,\Delta)\|_\infty <1$.

\vspace{0.1in}
\noindent
\emph{B)$\to$ A):} Assume A) fails to hold and we will show that B) fails
to hold.

If A) fails to hold then there is a frequency $\theta_0\in[0,2\pi]$ such that:
\begin{align*}
\inf_{D>0} \,\, \bar{\sigma} \left( 
   \bsmtx D \cdot I_{n_v} & 0 \\ 0 & I_{n_e} \esmtx
  M(e^{j\theta_0}) \bsmtx D^{-1} \cdot I_{n_w} & 0 \\ 0 & I_{n_d} \esmtx
   \right) \ge  1.
\end{align*}
Thus, Condition A) in Lemma~\ref{lem:MatrixSSV} fails to holds with
the complex matrix $M(e^{j\theta_0})$. By contraposition of Condition C)
in Lemma~\ref{lem:MatrixSSV}, there exists a matrix
$\Delta_0 \in \C^{n_w \times n_v}$ such that: (i)
$\bar{\sigma}(\Delta_0)\le 1$ and (ii) either
$F_U(M(e^{j\theta_0}),\Delta_0)$ is not well-posed or
$\bar{\sigma}(F_U(M(e^{j\theta_0}),\Delta_0)) \ge 1$. 

The offending matrix $\Delta_0$ can be interpolated to construct an
offending system $\Delta(z)$. Specifically, if the frequency is
$\theta_0\in (0,\pi) \cup (\pi,2\pi)$ then we can construct a stable,
LTI system $\Delta$ such that $\Delta(e^{j\theta_0})=\Delta_0$ and
$\|\Delta\|_\infty \le 1$. Theorem 9.1 of \cite{zhou96} gives the
construction for continuous-time and a similar construction works in
discrete-time.  One additional detail is required if the frequency is
$\theta_0=0$, $\pi$, or $2\pi$. ($M(e^{j\theta_0})$ is the
  same for $\theta_0=0$ and $2\pi$.) The frequency response
$M(e^{j\theta_0})$ is a real matrix at these frequencies.  In this
case, it is possible to construct $\Delta_0$ to be a real (rather than
complex) matrix.  This follows from a result for the structured
singular value (Section 9.7 of \cite{packard93}).  Thus if
$\theta_0=0$, $\pi$, or $2\pi$ then we can define a constant system
$\Delta(z)=\Delta_0$ with
$\|\Delta\|_\infty = \bar{\sigma}(\Delta_0)\le 1$.

There are two possibilities for the remainder of the proof. If
$F_U(M(e^{j\theta_0}),\Delta_0)$ is not well-posed then
$\det(I-M(e^{j\theta_0})\Delta(e^{j\theta_0}))=0$.  This implies the
system $\Delta$ causes $F_U(M,\Delta)$ to be unstable with a pole on
the unit disk at $z=e^{j\theta_0}$.  Alternatively, if
$\bar{\sigma}(F_U(M(e^{j\theta_0}),\Delta_0)) \ge 1$ then the system
$F_U(M,\Delta)$ satisfies:
\begin{align*}
  \| F_U(M,\Delta)\|_\infty \ge
        \bar{\sigma}\left(F_U(M(e^{j\theta_0}),\Delta(e^{j\theta_0}) )\right) \ge 1.
\end{align*}
In either possibility, Condition B) of Lemma~\ref{lem:SystemSSV} fails
to hold.
\end{proof}

\end{document}